\documentclass[11pt]{amsart}
\usepackage{amssymb}
\usepackage{charter}
\usepackage{url}
\usepackage{hyperref}
\usepackage{graphicx}
\usepackage{xcolor}
\usepackage{float}
\usepackage{booktabs}

\hypersetup{pdfpagemode=UseNone}

\newtheorem{theorem}{Theorem}[section]
\newtheorem{lemma}[theorem]{Lemma}
\newtheorem{proposition}[theorem]{Proposition}
\newtheorem{corollary}[theorem]{Corollary}
\newtheorem{notation}[theorem]{Notation}
\newtheorem{conjecture}[theorem]{Conjecture}

\theoremstyle{remark}
\newtheorem{remark}[theorem]{Remark}

\theoremstyle{definition}
\newtheorem{definition}[theorem]{Definition}

\theoremstyle{example}

\newcommand{\Sym}{\mathrm{Sym}}
\newcommand{\Alt}{\mathrm{Alt}}

\newcommand{\F}{\mathbb{F}}
\newcommand{\Z}{\mathbb{Z}}
\newcommand{\Q}{\mathbb{Q}}
\newcommand{\C}{\mathbb{C}}
\newcommand{\GL}{\mathrm{GL}}
\newcommand{\SL}{\mathrm{SL}}
\newcommand{\PGL}{\mathrm{PGL}}
\newcommand{\PSL}{\mathrm{PSL}}

\newcommand{\Di}{\mathrm{D}}
\newcommand{\Mo}{\mathrm{M}}

\newcommand{\Aut}{\mathrm{Aut}}

\newcommand{\thickbar}[1]{\mathbf{\bar{\text{$#1$}}}}

\newcommand{\scst}{\scriptscriptstyle}

\begin{document}

\title[Finite monomial linear groups of prime degree]
{Classifying finite monomial linear groups of
prime degree in characteristic zero}

\begin{abstract}
Let $p$ be a prime and let $\C$ be the complex field. 
We explicitly classify the finite solvable 
irreducible monomial subgroups of $\GL(p,\C)$ up to 
conjugacy. That is, we give a complete and irredundant 
list of $\GL(p,\C)$-conjugacy class representatives 
as generating sets of monomial matrices. 
Copious structural information about 
non-solvable finite 
irreducible monomial subgroups of $\GL(p,\C)$ is
also proved, enabling a classification of all
such groups bar one family. We explain the obstacles 
in that exceptional case. For $p\leq 3$, we classify 
all finite irreducible subgroups of $\GL(p,\C)$. 
Our classifications are available publicly 
in {\sc Magma}. 
\end{abstract}

\author{Z.~B\'{a}cskai}
\address{Defence Science Technology Group, 
 Canberra, Australia}
\email{zoltan.bacskai@defence.gov.au}

\author{D.~L.~Flannery}
\address{School of Mathematics, Statistics 
and Applied Mathematics,
National University of Ireland, Galway, Ireland}
\email{dane.flannery@nuigalway.ie}

\author{E.~A.~O'Brien}
\address{Department of Mathematics,
University of Auckland,
Private Bag 92019,
Auckland,
New Zealand}
\email{e.obrien@auckland.ac.nz}

%\subjclass[2010]{20H20, 20E99, 20-04}

\dedicatory{In memory of L.~G.~Kov\'{a}cs}

\maketitle

\section{Introduction}\label{Introduction}

Classifying finite subgroups of $\GL(n,\F)$ 
for various fields $\F$ and degrees $n$
is an enduring problem  
in linear group theory.
Early results are due to Jordan, Klein, 
Maschke, and Schur. Subsequently contributions
 were made by Dickson, Blichfeldt, Brauer, and 
Feit, to name just a few. 

Special degrees (`small', prime, product of 
two primes) have received the most attention. 
We focus on prime degree $p$, which eases the 
workload somewhat. For example, 
an irreducible subgroup of $\GL(p,\F)$ can
be imprimitive in only one way (monomial). 
Furthermore, classifications in prime degrees 
may be needed to classify groups of 
composite degree.

The term \emph{classify} has disparate meanings in  
linear group theory. By \emph{classification}, we
mean a list of groups of the declared kind
in $\GL(n,\F)$ that contains every group 
of that kind exactly once up to $\GL(n,\F)$-conjugacy; 
also, each listed group is given explicitly,
as a generating set of matrices. 

D.~A.~Suprunenko classified several kinds of 
linear groups over finite and infinite fields. 
In \cite[Theorem~6, p.~167]{SuprunenkoI}
and \cite[\S\! 22.1]{SuprunenkoI}, 
the maximal irreducible solvable linear groups of prime 
degree over finite fields and algebraically closed fields
are listed up to conjugacy.
This was extended to other classical groups by 
Detinko~\cite{DetinkoClassic,DetinkoThesis}.
Many more classifications of finite linear groups
have been published, some of which are surveyed in 
\cite[\S\! 8]{FeitSurvey}, \cite{TiepZalesskii},
and \cite[\S \S \! 4.5--4.7]{ZalesskiiII}. 

For our classifications, we 
take $\F$ to be the complex field $\C$.
The list of groups in each degree is thus 
infinite. By contrast, the number of conjugacy classes of 
finite primitive subgroups of $\SL(n,\C)$ is 
finite (hence the popularity of these restrictions 
in the literature). 
Each group in our lists has a unique label: an integer 
parameter string that specifies the non-zero entries of 
matrix generators. 

In \cite{Conlon,Conlon2} Conlon classified the
non-abelian finite $p$-subgroups of classical groups 
of degree $p$ over a field of characteristic not $p$.
These papers set a benchmark of thoroughness, demonstrating 
that imprimitive groups in the full general linear 
group could be handled without too much difficulty.

L.~G.~Kov\'{a}cs initiated and 
guided a research program aimed at classifying finite 
linear groups to the standard of \cite{Conlon,Conlon2}. 
Work within this program includes 
that by B\'{a}cskai~\cite{Bacskai}, 
Flannery~\cite{FlanneryMemoir}, 
H\"{o}fling~\cite{Hoefling}, 
Short~\cite{Short}, and Sim~\cite{Sim}.
Our paper is a development of \cite{Bacskai}.
We emphasize again the scale of all these classifications: 
they are complete and irredundant up to conjugacy
in the relevant $\GL(n,\F)$, 
with representatives given by generating sets of matrices.

One motivation for the Kov\'{a}cs program has origins
 in computational group theory. 
A certain maximal subgroups algorithm proposed by 
Kov\'{a}cs, Neub\"{u}ser, and Newman requires
lists of irreducible linear groups 
over finite fields, based on an equivalence 
with primitive permutation groups 
(see \cite[pp.~2--4]{Short}, and 
 \cite{Couttsetal,EickHoefling} for 
later progress in this direction). 
Classifications of finite linear groups over
$\C$ also serve as a resource for 
classifying linear groups over finite 
fields (see the use of \cite{Flannery2} 
in \cite{Flannery4FiniteFields}).

A comparable classification, of finite non-solvable 
irreducible monomial subgroups of $\SL(p,\C)$, was
achieved by Dixon and Zalesskii~\cite{DZII}. 
While our work inevitably has overlaps 
with \cite{DZII}, there are significant differences.
First, a classification up to conjugacy
in $\GL(p,\C)$ is remote from an 
analogous classification in $\SL(p,\C)$; 
the restriction to $\SL(p,\C)$ affords various 
simplifications that are not applicable in the full 
general linear group.  Moreover, we 
have completely classified the solvable finite 
irreducible monomial subgroups of $\GL(p,\C)$ for 
all $p$---a notable accomplishment in its own right. 
The non-solvable case lacks a complete solution for 
arbitrary $p$, as we explain in Section~\ref{Overview}.  

Classifications such as the ones in this paper
are dense with intricacies that may increase
the likelihood of error. To address 
this issue, we have made the classifications 
publicly available as part of the computer 
algebra system {\sc Magma}~\cite{Magma};
they can be incorporated into other systems. 
Output is a list of groups over an algebraic 
number field prescribed by an input bound on group 
order. Measures to verify correctness
are discussed in Section~\ref{Implementation}.

Initially we deal with solvable monomial subgroups 
of $\GL(p,\C)$. 
The treatment is then widened to non-solvable groups, 
culminating in a classification of the finite irreducible 
monomial subgroups of $\GL(p,\C)$ for $p\leq 11$.
Additionally, for $p\leq 3$, we classify 
all finite irreducible subgroups of 
$\GL(p,\C)$. To classify finite non-solvable 
primitive subgroups of $\GL(p,\C)$ for $p>3$, 
one might utilize the description in \cite{DZI} of  
the finite primitive subgroups of 
$\SL(p,\C)$ up to isomorphism. 

As noted above, our development 
of the classification and its exposition 
follow \cite{Bacskai}.   
Recently, the second and third authors
 revisited the topic, with a long-held aim of making 
these results accessible to the wider research 
community. We discovered errors in \cite{Bacskai} 
which impact on the correctness of 
that work. Consequently, we prepared a new
self-contained account that resolves these 
errors, and, for the first time, provides the
classification in a format suitable for 
further computation.

\section{Preliminaries}

Unless stated otherwise, $\F$ denotes 
an arbitrary field.
The group $\Mo(n,\F)$ of all monomial matrices in 
$\GL(n,\F)$ splits over the subgroup $\mathrm{D}(n,\F)$ 
of diagonal matrices:
$\Mo(n,\F) = \Di(n,\F) \rtimes \mathrm{P}(n)$, where 
$\mathrm{P}(n)$ is the group of 
permutation matrices. We identify $\mathrm{P}(n)$ with 
$\mathrm{Sym}(n)$; say, under the isomorphism that maps 
$\alpha \in \mathrm{Sym}(n)$ 
to $[\delta_{i\alpha, j}]_{i,j}\in \mathrm{P}(n)$ 
where $\delta$ is the Kronecker delta.
Each $G\leq \mathrm{M}(n,\F)$ is an 
extension of its \emph{diagonal subgroup} 
$G\cap \allowbreak \mathrm{D}(n,\F)$ by its 
\emph{permutation part}
$\mathrm{D}(n,\F) G \cap \allowbreak
\mathrm{P}(n)$.
\begin{definition}
Let $\phi$ be the natural surjection 
$\mathrm{M}(n,\F)\rightarrow \mathrm{Sym}(n)$ 
defined by $\phi \colon dt \mapsto t$ for 
$d\in \allowbreak \mathrm{D}(n,\F)$ and 
$t\in \mathrm{P}(n)$.
\end{definition} 
Note that $G$ has permutation part $\phi(G)$.
We speak of the diagonal subgroup $G\cap \ker \phi$
as a \emph{$\phi(G)$-module}; 
$x\in \mathrm{M}(n,\F)$ acts by conjugation on 
$\mathrm{D}(n,\F)$ as $\phi(x)$ does, permuting 
diagonal entries as $\phi(x)$ permutes 
$\{ 1, \ldots , n\}$.
\begin{lemma}
If $G\leq \mathrm{M}(n,\F)$ is irreducible then 
$\phi(G)$ is transitive.
\end{lemma}

So we must first solve a classification problem for  
permutation groups: classify the transitive 
groups of prime degree.
\begin{notation}
\emph{$\widetilde{\Mo}(n,\F)$ denotes the
group of all $n\times n$ monomial matrices 
over the roots of unity in $\F$. }
\end{notation}

\begin{lemma}\label{EntriesAreTorsion}
If $G\leq \mathrm{M}(n,\F)$ is finite and 
$\phi(G)$ is transitive then $G$ is 
$\Mo(n,\F)$-conjugate to a subgroup 
of $\widetilde{\Mo}(n,\F)$.
\end{lemma}
\begin{proof}
Let $e_k\in \F^n$ be the vector with $1$ in position 
$k$ and $0$s elsewhere. The orbit $Ge_1$ contains a basis 
$\{b_1,\ldots, b_n\}$ of $\F^n$. For each $g\in G$ and 
$i$, we have $gb_i = \lambda b_j$ for some $j$ and
$\lambda\in \F^\times$. Since $\lambda e_1\in Ge_1$, 
the scalar $\lambda$ has finite order. Thus, if 
$b\in \Mo(n,\F)$ is the matrix with $k$th column $b_k$, 
then $G^b\leq \widetilde{\Mo}(n,\F)$.
\end{proof}

In view of Lemma~\ref{EntriesAreTorsion}, we classify
subgroups of $\widetilde{\mathrm{M}}(p,\C)$. 
That is, listed groups will be given
by generating sets of monomial matrices,  and 
the non-zero entries of each generator are roots of unity.

The key steps in our approach are as follows.
\begin{itemize}
\item[(1)] 
Classify the transitive $T\leq \mathrm{Sym}(p)$ 
up to conjugacy.
\vspace{2pt}
\item[(2)] For each permutation part $T$, list the 
candidate diagonal subgroups $A$, i.e.,
the finite $T$-submodules $A$ of $\Di(p,\F)$.
\vspace{2pt}
\item[(3)] For pairs $(T,A)$
drawn from (1) and (2), solve the extension problem  
in $\mathrm{M}(p,\F)$ up to $\GL(p,\F)$-conjugacy,  
ensuring that each retained extension 
of $A$ by $T$ is irreducible.
\vspace{2pt}
\item[(4)] Eliminate $\GL(p,\F)$-conjugacy among all 
subgroups of $\mathrm{M}(p,\F)$ found in (3).
\end{itemize}

\noindent
These steps are carried out in 
Section~\ref{TransitiveSubgroupsOfSymp};
Section~\ref{SubgroupsOfDpC};
Sections~\ref{IrreducibilityCriteria} and
 \ref{MonomialGroupsWithCp}--\ref{DegreesGreaterThan5};
and Sections~\ref{ConjugacyResults} and
\ref{MonomialGroupsWithCp}--\ref{DegreesGreaterThan5}, 
respectively.

 The final lists are complete (every 
finite irreducible monomial subgroup of $\GL(p,\F)$ 
is represented) and irredundant 
(no $\GL(p,\F)$-conjugacy class is represented 
more than once).  
We justify completeness 
by proving that step~(4)
does not remove any conjugacy class. 
Most of the taxonomic
complication arises in step~(2),
especially when $T$ is solvable.

\subsection{Transitive subgroups of 
 $\mathrm{Sym}(p)$}
\label{TransitiveSubgroupsOfSymp}

Permutation groups of prime degree have been 
studied from the time of Galois.
We recap some of the essential theory.

A transitive group $T\leq \mathrm{Sym}(p)$ has a unique
simple normal transitive subgroup $U$. The quotient 
$N_{\mathrm{Sym}(p)}(U)/U$ is cyclic of order dividing $p-1$.
Thus $T$ is solvable if and only if $U$ is solvable, i.e.,
$|U|=p$.
\begin{notation}\label{stuNotation}
\emph{Throughout, $s=
(1,2,\ldots, p)\in \mathrm{Sym}(p)$. 
Let $t \in \allowbreak \mathrm{Sym}(p)$ be defined by 
$i\mapsto iu \bmod p$ where $u$ is the least 
primitive element modulo $p$ in $\{ 1, \ldots, p-1\}$. 
So $|s|=p$, $|t|=p-1$, and $s^t = s^u$.}
\end{notation}

Since $T$ contains a conjugate of $s$ in $\mathrm{Sym}(p)$, 
we assume that $s\in T$ henceforth.
\begin{lemma}\label{TransitivePrimeDegreePermGroups}
If $T$ is solvable then $T\leq N_{\Sym(p)}(\langle s\rangle)=
\langle s, t \rangle \cong C_p\rtimes C_{p-1}$. 
\end{lemma}
A solvable irreducible monomial subgroup of 
$\mathrm{GL}(p,\F)$ is therefore conjugate to a 
subgroup of 
$\mathrm{M}(p,\F)$ with permutation part 
$\langle s, t^a\rangle$ where $a\, | \, (p-1)$.
On the other hand, $\langle s\rangle$ is 
not normal in non-solvable $T$.
\begin{proposition}\label{FitTransitiveDegreep}
A non-trivial normal subgroup of $T$ is transitive.
Hence the Fitting subgroup $\mathrm{Fit}(T)$ is non-trivial 
if and only if $T$ is solvable, in which case 
$|\mathrm{Fit}(T)|=p$. 
\end{proposition}

Table \ref{table-nmr-reps} 
(extracted from \cite[Table~1]{Neumann};
see also \cite[\S \! 1.1]{DZII})
displays facts about all non-solvable $U$.
\begin{table}[h]
\begin{center}
\begin{tabular}{|c|c|c|c|}\hline
 $U$ & $N_{\mathrm{Sym}(p)}(U)$ & \mbox{Degree} & 
\mbox{\# rep.s} \\
\hline\hline
$\mathrm{Alt}(p)$ &  $\mathrm{Sym}(p)$ & $p\geq 7$ & 1 \\
\hline
$\SL(d,q)$ &  
$\Sigma \mathrm{L}(d,q)$ 
& $p=\frac{q^d-1}{q-1}$ & 
$\begin{array}{c}
1 \, (d=2)\\  [-.5mm]
2 \, (d \geq 3) 
\end{array}$ \\
\hline
$M_p$ &  $M_p$ & $p=11$, $23$ & $1$ \\ \hline
$\PSL(2,11)$ &  $\PSL(2,11)$ & $p=11$ & $2$\\ \hline
\end{tabular}
 \end{center}
\medskip
\caption{Non-abelian simple transitive 
permutation groups of degree $p$}
\label{table-nmr-reps}
\end{table}

\noindent 
The fourth column states the number of 
inequivalent faithful representations of $U$ in
$\mathrm{Sym}(p)$; $M_{11}$ and $M_{23}$ are 
Mathieu groups; $d$ is prime and 
$\mathrm{gcd}(d,q-1)=1$, so $\SL(d,q)\cong \PSL(d,q)$. 
The normalizer $\Sigma \mathrm{L}(d,q)$ is 
$\SL(d,q)\rtimes \mathrm{Aut}(\F_q)$ where $\F_q$ 
denotes the field of size $q$.
Of course, $\mathrm{Alt}(5)$ appears as $U$ twice.

Observe that $\mathrm{Sym}(p)$ always has 
transitive subgroups of the following kinds:
the solvable ones, $\mathrm{Alt}(p)$, 
and $\mathrm{Sym}(p)$; these we 
call \emph{compulsory}.
All but three non-compulsory transitive 
permutation groups of degree $p$ belong to 
the `projective' family (i.e., with $U$ as in row~2 of 
Table~\ref{table-nmr-reps}). 
Bateman and Stemmler~\cite[Theorem~4]{BS} show that for 
large $n$ there are at most $50 \sqrt{n} /(\log n)^{2}$ 
primes of this form not exceeding $n$. So
there are infinitely many `non-projective' primes.

The next fact does not seem to be widely known. 
\begin{theorem}
\label{IsoQDegreep}
Transitive subgroups of $\mathrm{Sym}(p)$ 
are conjugate if and 
only if they are isomorphic.
\end{theorem} 
\begin{proof}
If $U$ is $C_p$, $\mathrm{Alt}(p)$, $\SL(2,q)$, or
$M_{p}$, then there is only one faithful 
permutation representation of $U$ of degree $p$ 
up to equivalence, and hence a single conjugacy 
class of groups in $\mathrm{Sym}(p)$
isomorphic to $U$.

Suppose that $U\cong \PSL(2,11)$ and $\theta$ is 
a transitive embedding of $U$ in 
$\mathrm{Sym}(p)$ inequivalent to $\mathrm{id}_U$. 
If $\alpha \in \mathrm{Aut}(U)$ and 
$\theta\alpha(u)^a = \theta(u)$
for some $a\in \mathrm{Sym}(p)$ and all 
$u\in U$, then $a\in \theta(U)$ because
$\theta(U)$ is self-normalizing. Since 
$\mathrm{Out}(U)\neq 1$, there will be an $\alpha$ 
such that $\mathrm{id}_U$ and $\theta\alpha$ are 
equivalent; whence $U$ is conjugate to $\theta(U)$.

If $U\cong \SL(d,q)$ for $d\geq 3$ then  
the graph (inverse transpose) 
automorphism of $U$ swaps the two inequivalent 
transitive representations of $U$ in 
$\mathrm{Sym}(p)$~\cite[p.~523]{Neumann}.

Now let $S$ and $\hat{S}$ be isomorphic transitive
subgroups of $\mathrm{Sym}(p)$, with simple normal 
transitive subgroups $U$, $\hat{U}$ respectively.
By the above, $U^w = \hat{U}$ for some 
$w\in \mathrm{Sym}(p)$. Thus $S^w$ normalizes $\hat{U}$.
Since $N_{\mathrm{Sym}(p)}(\hat{U})/\hat{U}$ 
is cyclic, $S^w = \hat{S}$.
\end{proof}

\subsection{Conjugacy}
\label{ConjugacyResults}

Already we witness a division of the classification 
into mutually disjoint families: 
groups with permutation parts that are not 
conjugate in $\mathrm{Sym}(p)$ are not 
$\mathrm{M}(p,\F)$-conjugate. 
Such groups cannot even be isomorphic. 
\begin{theorem}
\label{GLConjugateImpliesSamePP}
Suppose that $G, H\leq \Mo(p,\F)$ are isomorphic, 
with $\phi(G)$ and $\phi(H)$ transitive. Then $\phi(G)$ 
and $\phi(H)$ are $\mathrm{Sym}(p)$-conjugate.
If, furthermore, $G$ is non-solvable, then 
$\Di(p,\F)\cap G$ maps onto $\Di(p,\F)\cap H$ under 
any isomorphism $G\rightarrow H$.
\end{theorem}
\begin{proof}
The theorem follows (with a little effort)
from Lemma~\ref{TransitivePrimeDegreePermGroups}, 
Proposition~\ref{FitTransitiveDegreep}, and
Theorem~\ref{IsoQDegreep}.
\end{proof}

Theorem~\ref{GLConjugateImpliesSamePP} is vital
in our solution of the conjugacy problem: groups with 
different permutation parts are not conjugate, and 
conjugacy that does not respect diagonal subgroups 
can only occur between (irreducible) solvable groups.
\begin{theorem}
\label{Protoconjugacy}
Let $G$ be an irreducible subgroup of 
$\widetilde{\mathrm{M}}(p,\F)$ with 
non-scalar diagonal subgroup $A$. Suppose that 
$G^w \leq \widetilde{\mathrm{M}}(p,\F)$ 
and $A^w \leq \mathrm{D}(p,\F)$ for 
some $w\in \GL(p,\F)$. 
Then $w$ is monomial; furthermore, 
$w\in \widetilde{\Mo}(p,\F)$ up to scalars if
$\F$ is algebraically closed.
\end{theorem}
\begin{proof}
By Clifford's Theorem, the $A$-submodules of $\F^n$ 
are exactly the $1$-dimensional subspaces $\F e_i$.
Then $we_i\in \F e_{j}$
for some $j$, because $A^w e_i\subseteq \F e_i$. 
Hence $w$ is monomial. We may
assume that $w$ is diagonal.  
Fix $g\in G$; the map  defined by $b\mapsto [b,g]$ 
is an endomorphism of $\Di(p,\F)$. 
Since $[w,g]\in  \widetilde{\Mo}(p,\F)$,
we have $[w^n,g] = [w,g]^n=1$ for some $n$ 
and all $g$. By Schur's Lemma, $w^n$ is scalar. 
Taking $n$th roots, we see that
a scalar multiple of $w$ has finite order.
\end{proof}

In other words, $\GL(p,\C)$-conjugacy 
that respects (non-scalar) diagonal subgroups is 
effected by a monomial matrix.
So our priority is to sort out the conjugacy 
classes of $\widetilde{\mathrm{M}}(p,\C)$. 
Moreover, Theorems~\ref{GLConjugateImpliesSamePP} 
and \ref{Protoconjugacy} are frequently used to 
show that $\GL(p,\C)$-conjugacy
among non-solvable groups is the same as
$\widetilde{\mathrm{M}}(p,\C)$-conjugacy.
\begin{remark}
 The \emph{isomorphism question} 
for a set $\mathcal S$ of subgroups of  
$\GL(n,\F)$ asks: if $G$, $H\in \mathcal S$ 
are (abstractly) isomorphic, are $G$ and $H$
linearly isomorphic ($\GL(n,\F)$-conjugate)? 
The answer is ``yes'' for
non-abelian finite $p$-subgroups  
of $\GL(p,\C)$ by \cite[Proposition~4.2]{Conlon},
but ``no'' more generally for 
finite irreducible subgroups of $\Mo(p,\C)$.
Cf.~Theorem~\ref{IsoQDegreep}. 
Corollary~\ref{SingleCClNonsolvable} 
gives another answer to the isomorphism 
 question. 
\end{remark}

\subsection{Irreducibility}
\label{IrreducibilityCriteria}

We prove the next theorem 
using Ito's result that the irreducible 
ordinary character degrees of a finite 
group divide the 
index of each abelian normal subgroup. 
\begin{theorem}
\label{NonScalarIffIrreducible}
Let $G$ be a finite subgroup of 
$\mathrm{M}(p,\C)$
such that $\phi(G)$ is transitive.
If $\Di(p, \C)\cap G$ 
is non-scalar, then $G$ is irreducible. 
Conversely, if $G$ is solvable irreducible, 
then $\Di(p,\C)\cap G$ is non-scalar. 
\end{theorem}

\begin{corollary}
If $G$ is a finite irreducible subgroup 
 of $\Mo(p,\C)$ then 
 $\Di(p,\C)\cap G$ is a maximal abelian 
normal subgroup of $G$. 
\end{corollary} 

\begin{remark}
Let $G = A\rtimes T$ for scalar 
$A\leq \mathrm{D}(n,\F)$ and $T\leq \mathrm{Sym}(n)$. 
Then $G$ fixes the all $1$s vector, so is reducible.
\end{remark}

\begin{remark}
The general converse of the first claim 
in Theorem~\ref{NonScalarIffIrreducible}
is false: $\mathrm{M}(5,\C)$ has irreducible 
subgroups isomorphic to $\mathrm{Alt}(5)$.
\end{remark}

\section{Diagonal subgroups}
\label{SubgroupsOfDpC}

Let $T\leq \mathrm{Sym}(p)$.
A finite $T$-submodule of $\mathrm{D}(p,\C)$ is the 
direct product of its Sylow $p$-subgroup and its 
Hall $p'$-subgroup and each is a $T$-module.
The submodule listing problem 
bifurcates accordingly.

\subsection{$\langle s\rangle$-modules}

\subsubsection{The modules of $p$-power order}
\label{pPowerSubmodules}

Our paradigm for listing the $\langle s\rangle$-submodules 
of $\Di(p,\C)$ of $p$-power order is \cite[\S \! 1]{Conlon}.

Note that $\Di(p,\C)$ is a central product 
$XZ$ amalgamating the scalar subgroup of order $p$, 
where $Z$ is the group of all scalars and 
$X = \SL(p,\C)\cap \Di(p,\C)$.
We define endomorphisms $\gamma$ and $\chi$ of 
$\mathrm{D}(p,\C)$
by
\[
\gamma(d) = d^{1-s}, \qquad 
 \chi(d) = d^{1+s+\cdots+s^{p-1}}. 
\]
Then $Z=\chi(Z)=\chi(\mathrm{D}(p,\C))=
\ker \, \gamma$ and 
$X = \gamma (X) = \gamma(\mathrm{D}(p,\C))
=\ker \, \chi$.  
For $\mathrm{i} = \allowbreak \sqrt{-1}$, let
\[
b_m = \mathrm{diag}(e^{2\pi{\rm i}/m}, e^{-2\pi{\rm i}/m},1,
 \ldots , 1), 
 \qquad
z_m = \mathrm{diag}(e^{2\pi{\rm i}/m}, \ldots , 
e^{2\pi{\rm i}/m}).
\]
\begin{lemma}\label{ZModules}
$Z$ has a unique $\langle s \rangle$-submodule 
of each order $m$, namely 
$Z_m=\langle z_m\rangle$. 
\end{lemma}

\begin{notation}
\emph{Let $D$ be the torsion subgroup of $\Di(p,\C)$.
If $\pi$ is a set of primes, then
 $A_{\pi}$ denotes the Hall $\pi$-subgroup of
$A\leq D$. So we write $A_{\{q\}}$ for the 
Sylow $q$-subgroup of $A$ if $\pi$ has a single 
prime $q$.  
The complement of $A_{\pi}$ in $A$ is denoted 
$A_{\pi'}$.}
\end{notation}

\begin{remark}
The scalar subgroup $Z_p= X\cap Z$ of order $p$
lies in every non-identity 
$\langle s\rangle$-submodule of $D_{\{p\}}$. 
\end{remark}

Next we determine the finite $\langle s\rangle$-submodules 
of $X_{\{ p\}}$.
\begin{definition}
For a positive integer $j$, let $n$, $m$ be the 
non-negative integers such that $m<p-1$ and $j = n(p-1)-m$. 
Define 
\[
X_{p^j} = X \cap (\ker \, \gamma^j), \qquad 
x_{p^j} = \gamma^m(b_{p^n}).
\]
\end{definition}
\begin{remark}
$z_{p^{j+1}}^p = z_{p^j}$ and 
$x_{p^{j+p-1}}^p = x_{p^j}$.
\end{remark}

\begin{notation}
{\em If $A$ is an abelian $q$-group, then 
$\Omega_k A := \{ a\in A \; |\; \allowbreak
 a^{q^k} = 1\}$, i.e., the largest subgroup 
of exponent at most $q^k$.}
\end{notation}

\begin{lemma}
\label{XModules}
$X_{\{p\}}$ is a uniserial 
$\langle s\rangle$-module; 
it has a unique submodule 
at every order $p^j$, namely $X_{p^j}$,
generated as an 
$\langle s\rangle$-module by $x_{p^j}$.
\end{lemma}
\begin{proof}
Suppose that $|X_{p^j}| = p^j$ for 
$j\geq 1$. The $\langle s\rangle$-epimorphism 
$\gamma \colon X_{p^{j+1}}\rightarrow X_{p^j}$ has kernel 
$X_p$. Thus $|X_{p^{j+1}}| = p |X_{p^j}|
= p^{j+1}$, proving that $|X_{p^j}| = p^j$  
for all $j$ by induction.

Let $B\leq X_{\{p\}}$ be a non-identity 
$\langle s\rangle$-submodule.
So $B/X_p$ is an $\langle s\rangle$-submodule of 
$X_{\{p\}}/X_p \cong X_{\{p\}}$. If $B\neq X_p$ 
then we replace $B$ by $B/X_p$ in $X_{\{p\}}/X_p$ 
and repeat. The recursion eventually
terminates, at which point $B= X_{p^j}$ for some $j$. 

The $\langle s\rangle$-module generated by $b_{p^n}$ is
$\Omega_n X_{\{p\}} = X_{p^{n(p-1)}}$.
Hence $\gamma^m (b_{p^n})$ generates 
$\gamma^m(X_{p^{n(p-1)}})=X_{p^{n(p-1)-m}}$.
\end{proof}

We obtain all finite $\langle s \rangle$-submodules 
of the Sylow $p$-subgroup 
\[
D_{\{p\}}/Z_p = X_{\{p\}}/Z_p\times Z_{\{ p\}}/Z_p
\] 
of $\Di(p,\C)/Z_p$ from Lemmas~\ref{ZModules} and 
\ref{XModules} 
and the following well-known theorem
(for a proof, see \cite[1.6.1, p.~35]{Schmidt}).
\begin{theorem}[Goursat--Remak]
\label{GoursatRemak}
Let $U$ and $V$ be $R$-modules for
an associative unital ring $R$. 
If $\theta$ is an $R$-isomorphism of a section
$U_1/U_2$ of $U$ onto a section $V_1/V_2$ of $V$, then
\[
W_\theta = \{ uv \; |\;  u\in U_1, \, v\in V_1, \, 
\theta(uU_2) = vV_2\}
\]
is an $R$-submodule of $U\times V$ such that 
\[
U_2 = W_\theta\cap U, \quad V_2= W_\theta\cap V, 
\quad U_1 = U\cap W_\theta V,\quad
V_1 = UW_\theta \cap V ,
\] 
and $W_\theta/U_2V_2 \cong U_1/U_2\cong V_1/V_2$.

Conversely, let $W$ be an $R$-submodule of $U\times V$. Put
\[
U_2 = W\cap U, \quad V_2= W\cap V, \quad U_1 = U\cap WV,\quad
V_1 = UW\cap V,
\] 
and define $\alpha\colon U_1/U_2\rightarrow V_1/V_2$
by $\alpha(uU_2) = vV_2$ where $v$ is any element
of $uW\cap V$. Then $\alpha$ is an $R$-isomorphism 
such that $W=W_\alpha$.
\end{theorem}

That is, apart from `Cartesian' submodules 
$X_{p^j} Z_{p^k}$,
the $\langle s \rangle$-submodules of $D_{\{ p\}}$
are in one-to-one correspondence with the 
$\langle s\rangle$-isomorphisms 
between non-identity sections of $X_{\{p\}}/Z_p$ 
and $Z_{\{p\}}/Z_p$.
\begin{definition}
For $k\geq 0$, let $Y_{0,k,0} = Z_{p^k}$. 
For $j,k\geq 1$ and $l\geq 0$, let
\[
Y_{j,k,l} = \langle x_{p^{j+1}}z_{p^{k+1}}^l, 
X_{p^j},Z_{p^k}\rangle .
\]
Define $\mathcal{Y}$ to be the set of all $Y_{j,k,l}$
where $0\leq l\leq p-1$ and 
either $j = l = 0$ and  $k\geq 0$, or
$j$, $k\geq 1$.
\end{definition}
\begin{theorem}[cf.{~\cite[1.8]{Conlon}}]
\label{AllpsModules}
$\mathcal Y$ is the set of all finite 
$\langle s \rangle$-submodules 
of $D_{\{p\}}$. 
\end{theorem}
\begin{proof}
As forecast, this is an application of
Theorem~\ref{GoursatRemak}, relying on 
Lemmas~\ref{ZModules} and \ref{XModules}.
\end{proof}
\begin{remark} \
\label{YGeneratorsAndOrder}
\begin{itemize}
\item[(i)]
If $(j,k,l)\neq (0,0,0)$ then $Y_{j,k,l}$ 
is generated as an $\langle s \rangle$-module by 
$x_{p^{j+1}}z_{p^{k+1}}^l$ and $z_{p^k}$. 
\item[(ii)] $|Y_{j,k,l}|=p^{j+k}$.
\item[(iii)] 
Each element of $\mathcal Y$ is labeled by a unique triple:  
$Y_{j,k,l}=Y_{a,b,c}\Rightarrow (j,k,l) = (a,b,c)$.
\end{itemize}
\end{remark}

\subsubsection{The modules of order coprime to $p$}
\label{ModulesOfpCoprimeOrder}

Let $q$ be a prime, $q\neq p$. Clearly
\[
D_{\{q\}} = X_{\{q\}}\times Z_{\{q\}}.
\]
As we will see, $X_{\{q\}}$ is a direct 
product of uniserial $\langle s\rangle$-submodules 
with no isomorphism between non-identity sections 
of the factors, and $\langle s\rangle$ acts 
non-trivially on every non-identity section of 
$X_{\{q\}}$. Hence, by the Goursat--Remak
Theorem,  all $\langle s\rangle$-submodules of 
$\Di(p,\C)$ of $q$-power order are Cartesian.

We gave a `closed' (submodule or 
subgroup) generating set of each finite 
$\langle s\rangle$-submodule of $D_{\{ p\}}$. It is 
infeasible to do the same for submodules of 
$D_{\{ q\}}$. A new feature is calculation with 
polynomials over the $q$-adic integers $\Z_q$ 
(the endomorphism ring of the quasicyclic $q$-group 
$C_{q^\infty}$, which acts on 
$D_{\{ q \}}\cong (C_{q^\infty})^p$ by extension 
of the $\Z$-action).
The complexity of these calculations 
varies with $q$ and $p$. We undertake
these calculations without imposing the `height' 
restriction of \cite[p.~366]{DZII}.
\begin{notation}
\emph{Denote reduction modulo $q$
by overlining. So $\overline{\Z}=\overline{\Z_q}=\F_q$, 
and a polynomial $f\in \Z_q[{\rm x}]$ or  $\Z[{\rm x}]$
maps to $\overline{\! f}\in \F_q[{\rm x}]$.} 
\end{notation}

We need the following version of Hensel's Lemma.
Its proof contains an algorithm that we use
to construct generators 
for submodules of $X_{\{q\}}$ 
(cf.~\cite[Lemma~12.8, p.~40]{FeitRepTh}).
\begin{lemma}\label{Hensellike}
Suppose that $f\in \Z_q[{\rm x}]$ is monic and 
$\, \overline{\! f}\in \F_q[{\rm x}]$ factorizes 
into the product of coprime monic polynomials 
$g_0$, $h_0\in \F_q[{\rm x}]$. 
Then there exist monic $g$, 
$h\in \Z_q[{\rm x}]$ such that
$\, \overline{\! g} = g_0$,
$\, \overline{\! h} = h_0$, and
$f= gh$.
\end{lemma}
\begin{proof} 
We call an integer polynomial {\em flat} if its 
coefficients lie in $\{0,1,\dots, \allowbreak q-1\}$. 
Each polynomial in $\Z_q[\rm x]$ is congruent 
modulo $q$ to a unique flat polynomial.

We have $a_0g_0+b_0 h_0 =1$ for some 
$a_0$, $b_0\in \F_q[\rm x]$.
Let $a$, $b$, $g_1$, $h_1\in \Z[{\rm x}]$ 
be the flat preimages of $a_0$, $b_0$, $g_0$,
$h_0$, respectively, modulo $q$.
Assume inductively that $n\geq 1$ and
$\overline{g_n} = g_0$, 
$\overline{h_n} = h_0$ for some $g_n$, 
$h_n\in \Z[{\rm x}]$ 
such that $f = g_nh_n + q^nc_n$
where $c_n \in \Z_q[{\rm x}]$. Then
$\mathrm{deg} \,c_n < \mathrm{deg} \, f$. Also, 
$ag_n + bh_n \equiv 1 \bmod q$.

By division in $\F_q[\rm x]$ and lifting, we 
get unique flat $v_n$, $w_n$ such that 
$bc_n \equiv w_ng_n + v_n \bmod q$ and 
$\mathrm{deg} \, v_n < \mathrm{deg} \, g_n$.
Let $u_n = ac_n + w_nh_n$, and let $u_n'$ be 
the unique flat polynomial congruent to 
$u_n$ modulo $q$. Then
\[
v_nh_n + u_ng_n = bc_n h_n + ac_n g_n \equiv c_n 
\bmod q,
\]
so $u_n'g_n\equiv c_n- v_nh_n \bmod q$. Since 
$\mathrm{deg} \, c_n <\mathrm{deg} \, f$
and $\mathrm{deg} \, v_n < \mathrm{deg} \, g_n$,
it follows that 
$\mathrm{deg} \, u_n' < \mathrm{deg} \, h_n$.

Define $g_{n+1} = g_n + q^n v_n$ and 
$h_{n+1} = h_n + q^n u_n'$. These polynomials
are monic, and
\begin{align*}
g_{n+1} h_{n+1} & \equiv g_n h_{n}
+q^n(u_{n}' g_{n} + v_nh_n )\\
& \equiv g_n h_{n}
+q^n c_{n} \\
& \equiv f \bmod q^{n+1}.
\end{align*}
Hence, by induction, for each positive 
integer $n$ there exist $g_n$, $h_n\in \Z[\rm x]$
such that $g_{n+1}\equiv g_n$,
$h_{n+1}\equiv h_n$, and $f\equiv g_n h_n \bmod q^n$. 
So the polynomial sequences $\{ g_n\}_{n\geq 1}$, 
$\{ h_n\}_{n\geq 1}$ 
converge in the $q$-adic sense to monic $g$, 
$h\in \Z_q[{\rm x}]$ such that $f=gh$.
\end{proof}

\begin{corollary}\label{IrredBelowIffIrredAbove}
Suppose that $f\in \Z_q[{\rm x}]$ is monic and 
$\overline{f}$ has no repeated roots (in 
any extension of $\F_q$). 
Then $f$ is $\Z_q$-irreducible if and only if 
$\, \overline{f}$ is $\F_q$-irreducible.
\end{corollary}
\begin{proof}
Since $\overline{f}$ has no repeated roots, it is 
coprime to its formal derivative. 
Thus $\overline{f}$ can only properly
factorize into a product of coprime polynomials, 
which contradicts
Lemma~\ref{Hensellike} if $f$ is irreducible.

For the converse, let $f_i\in \Z_q[{\rm x}]$ be the 
irreducible monic factors of $f$ in $\Z_q[{\rm x}]$.
Each $\overline{f_i}$ is monic and has no repeated 
roots. The previous paragraph implies that the 
$\overline{f_i}$ are $\F_q$-irreducible. 
\end{proof}

\begin{notation}\label{dAndvDefinitions}
\emph{$v := \frac{p-1}{d}$ where $d$ is the 
multiplicative order of $q$ modulo $p$.}
\end{notation}
\begin{definition}\label{Definefandgr}
From now on,
$f = f({\rm x}):=1+{\rm x} + \cdots + {\rm x}^{p-1}\in
\Z_q[{\rm x}]$. Let $g_1$ be an irreducible monic 
factor of $\, \overline{\! f}$.
For $1\leq r \leq v-1$, define
$g_{r+1} ({\rm x}) = 
\mathrm{gcd}(g_r({\rm x}^u), \overline{\! f}({\rm x}))$
where $u$ is the least primitive element modulo $p$ 
(see Notation~\ref{stuNotation}).
\end{definition}
\begin{remark}
We can impose a total ordering on $\F_q[{\rm x}]$
to ensure that the $g_r$ (and thus submodule generators 
in $X_{\{q\}}$) are canonically defined. 
\end{remark}

The polynomials $f$ and $\overline{f}$ do not have
repeated roots. Indeed, if $\xi$ is a root of 
$g_1$, then $\xi$ is a root of $\overline{f}$, 
hence a primitive $p$th root of unity, and 
the roots of $\overline{f}$ are the $\xi^i$.
\begin{proposition}\
\begin{itemize}
\item[{\rm (i)}] 
Each $g_r$ is $\F_q$-irreducible, and 
$\, \overline{\! f} = g_1\cdots g_v$.
\item[{\rm (ii)}] There are monic irreducible factors 
$f_1,\ldots , f_v\in \Z_q[{\rm x}]$ of $f$ such 
that $\overline{\! f_r}=g_r$ and $f = f_1\cdots f_v$. 
\end{itemize}
\end{proposition}
\begin{proof}
Denote the Frobenius automorphism of $\F_q(\xi)/\F_q$
by $\beta$. 
Let $\Xi_r=\{ \xi^{u^{1-r}q^i} \mid 
\allowbreak  1 \leq i \leq d\}$,
the orbit of $\xi^{u^{1-r}}$ under $\langle \beta\rangle$.
Further, let $\Lambda_r$ be the set of roots of $g_r$. 
We assert that $\Lambda_r = \allowbreak
\Xi_r$. Since $\langle \beta\rangle$ 
acts transitively on $\Xi_r$, this will prove that $g_r$ is 
irreducible; since the set of roots of $\overline{f}$ is 
partitioned by the $\Xi_r$, this will also prove that 
$\, \overline{\! f} = g_1 \cdots g_v$.

Certainly $\Lambda_1 = \Xi_1$, because $g_1$ is 
irreducible and has $\xi$ as a root. Assume inductively 
that $\Lambda_r = \Xi_r$ for some $r\geq 1$. Then 
$\xi^j\in \Lambda_{r+1}$ if and only if 
$\xi^{ju}\in \Lambda_r$. By the 
inductive hypothesis, this happens 
if and only if $j \equiv u^{1-(r+1)}q^i \bmod p$ 
for some $i$. Hence $\Lambda_{r+1} = 
\Xi_{r+1}$, completing the proof of (i) by induction.
As the polynomial rings are UFDs,  part~(ii) then 
follows from Corollary~\ref{IrredBelowIffIrredAbove}.
\end{proof}

Thus, we factorize $\, \overline{\! f}$ over $\F_q$, 
then lift to the irreducible factors $f_r\in \Z_q[{\rm x}]$ 
of $f$ by the algorithm in the proof of 
Lemma~\ref{Hensellike}. Although this factorization 
depends strongly on the value of $q$, we omit
$q$ in some of the attendant polynomial notation
to reduce clutter.

\begin{definition}\
\begin{itemize}
\item[(i)]
Let $X_{\{ q\}}^{(r)}$ be the set of  
elements of $X_{\{q\}}$ 
annihilated by $f_r(s)$.
\item[(ii)] 
$X_{q,n}^{(r)}:= \Omega_n X_{\{ q\}}^{(r)}$.
\item[(iii)] 
$f_{r'} := \prod_{j\neq r}f_j$.
\end{itemize}
\end{definition}
\begin{remark}
$X_{\{ q\}}^{(r)}$ is a 
$\Z_q\langle s\rangle$-submodule.
\end{remark}

\begin{proposition}\mbox{}
\label{410Bacskai}
\begin{itemize}
\item[{\rm (i)}] $X_{\{ q\}} = X_{\{ q\}}^{(1)}\times \cdots 
\times X_{\{ q\}}^{(v)}$.

\vspace{1pt}

\item[{\rm (ii)}] $X_{\{ q\}}^{(r)} = X_{\{ q\}}^{f_{r'}(s)}$
for $1\leq r \leq v$. 
\end{itemize} 
\end{proposition}
\begin{proof}
There exist $h_1, \ldots , h_v\in \Q_q[ {\rm x }]$ such that 
$\sum_{r=1}^vf_{r'}h_r=1$.  
We can choose $k\geq 0$ such that 
$c_r:=q^k h_r\in \Z_q[{\rm x}]$ for all $r$.
Since $X_{\{ q\}}^{q^k} = X_{\{q\}}$, 
\[
X_{\{q\}} = 
{\textstyle \prod}_{r=1}^v(X_{\{q\}})^{f_{r'}(s)c_r(s)}
 =
{\textstyle \prod}_{r=1}^v(X_{\{q\}}^{(r)})^{c_r(s)}
\subseteq 
{\textstyle \prod}_{r=1}^v X_{\{q\}}^{(r)} .
\]
Thus $X_{\{q\}}=\prod_{r=1}^vX_{\{q\}}^{(r)}$.

If $X_{q,1}^{(r)} \cap \Omega_1 
\big(\prod_{j\neq r} X_{\{q\}}^{(j)}\big) \neq 1$, then
there is a non-identity $b\in \allowbreak
\Omega_1  X_{\{q\}}$ such that $b=\allowbreak
b^{f_r(s)} =b^{f_{r'}(s)} =1$. 
But $f_r w + f_{r'}z \equiv 1 \bmod q$ for some
$w$, $z\in \Z_q[{\rm x}]$. Hence
$b= \allowbreak b^{f_r w(s)}b^{f_{r'} z(s)}=1$.
\end{proof}

We embark on the task of determining the finite 
$\langle s\rangle$-submodules of each $X_{\{ q\}}^{(r)}$.  
\begin{lemma}\label{411Bacskai}
$X_{q,1}^{(r)}$ is irreducible as an 
$\langle s\rangle$-module. 
\end{lemma}
\begin{proof}
Let $\mathrm{d}_r$ be the dimension of the subspace $X_{q,1}^{(r)}$ 
of the $(p-1)$-dimensional $\F_q$-space 
$\Omega_1X_{\{q\}}$. 
The conjugation action of 
$\langle s\rangle$ induces a linear transformation 
$s_r$ on $X_{q,1}^{(r)}$. 
Its minimal polynomial is $g_r$, so $d\leq \mathrm{d}_r$
by the Cayley-Hamilton theorem. This shows that
$\mathrm{d}_r =d$, because $\sum_r \mathrm{d}_r = dv$. 
Therefore $s_r$ has characteristic polynomial $g_r$.
Since $g_r$ is irreducible, $X_{q,1}^{(r)}$ is an
 irreducible $\langle s \rangle$-module.
\end{proof}

\begin{definition}
Let $f_{r,n}({\rm x})\in \Z[{\rm x}]$ be the $n$th 
approximation of $f_r$ found by the algorithm 
in the proof of Lemma~\ref{Hensellike} 
($f_{r,n+1}\equiv f_r \bmod q^n$).
Define
\[
f_{r',n}= 
{\textstyle \prod}_{j\neq r} f_{j,n}\in \Z[{\rm x}], 
\qquad \quad
x_{q,n}^{(r)} =b_{q^n}^{f_{r'}(s)}.
\]
Since $(x_{q,n}^{(r)})^{q^n} = 1$, we obtain
the useful working formula
$x_{q,n}^{(r)}= b_{q^n}^{f_{r',n}(s)}$. 
\end{definition}

\begin{lemma}
$X_{q,n}^{(r)}\cong (C_{q^n})^d$ is generated by 
$x_{q,n}^{(r)}$ as an $\langle s\rangle$-module. 
\end{lemma}
\begin{proof}
We saw in the proof of Lemma~\ref{411Bacskai} that 
$X_{q,n}^{(r)}\cong (C_{q^n})^d$. 
Also, $x_{q,n}^{(r)}$ generates
 $(\Omega_n X_{\{q\}})^{f_{r'}(s)} =
\Omega_n \big(X_{\{q\}}^{f_{r'}(s)}\big)
= X_{q,n}^{(r)}$ as an $\langle s \rangle$-module.
\end{proof}

\begin{proposition}
\label{StructureofXqr}
The $\langle s\rangle$-module $X_{\{ q\}}^{(r)}$ is 
uniserial: its only finite $\langle s\rangle$-submodules 
are the $X_{q,n}^{(r)}$, $n\geq 0$.
\end{proposition}
\begin{proof}
Cf.~the proof of Lemma~\ref{XModules}; here we use
Lemma~\ref{411Bacskai}.
\end{proof}

\begin{definition}
Let $\mathcal{W}_q$ be the set of all subgroups
$X^{(1)}_{q,n_1} X^{(2)}_{q,n_2} \cdots 
X^{(v)}_{q,n_v}Z_{q^c}$ 
of $D_{\{q\}}$ as $n_1, \ldots , n_v, c$ range 
over the non-negative integers.
\end{definition}

\begin{theorem}
\label{qsSubmodules}
$\mathcal{W}_q$ is the set of all finite 
$\langle s\rangle$-submodules of $D_{\{q\}}$. 
\end{theorem}
\begin{proof}
This is another application of 
Theorem~\ref{GoursatRemak}, relying now on 
Propositions~\ref{410Bacskai} and \ref{StructureofXqr}.
In particular, if $i\neq j$ then a non-identity 
section of $X_{\{q\}}^{(i)}$
is not $\langle s\rangle$-isomorphic to a non-identity 
section of $X_{\{q\}}^{(j)}$ (since $f_i(s)$ does 
not annihilate both sections). 
\end{proof}
\begin{remark}\label{CalWInjective}
Each element of $\mathcal{W}_q$ is labeled by a 
unique $(v+1)$-tuple $n_1, \ldots, n_v, c$.
\end{remark}

\subsection{All finite submodules of $\Di(p,\C)$}

Theorems~\ref{AllpsModules} and \ref{qsSubmodules} 
give the following.
\begin{theorem}
\label{AllModulesEverywhere}
The set $\mathcal A$ of direct products
$Y \times \Pi_q W_q$, where $Y\in \mathcal Y$ and
$W_q\in \mathcal{W}_q$ for finitely many primes $q\neq p$,
is the set of all finite $\langle s\rangle$-submodules 
of $\mathrm{D}(p,\C)$.
\end{theorem}

Theorem~\ref{AllModulesEverywhere} accounts for all the 
modules needed.
That is, the $T$-modules for 
$\langle s\rangle \leq T\leq \mathrm{Sym}(p)$ 
are listed by refinement of $\mathcal A$. 

As per Remarks~\ref{YGeneratorsAndOrder}~(iii) and 
\ref{CalWInjective}, we designate
each finite $\langle s \rangle$-submodule of $\Di(p,\C)$
by a unique integer parameter string. 

With the implementation in mind, we outline how to 
list all $\langle s\rangle$-submodules $M$ 
of a given order $o>1$.
Let $o=p^ab$ where $a\geq 0$ and $b$ is a positive 
integer not divisible by $p$. The possible Sylow 
$p$-subgroups of $M$ are the $Y_{j,k,l}$ where $j+k=a$, 
$0\leq l\leq p-1$, and either $j=l=0$ or $j,k\geq 1$. 
Let $q^e>1$ be the largest power of the prime $q$ in 
the prime factorization of $b$. Then 
$M_{\{p\}'}\cap X_{\{ q\}}$ is some
$W_q\in \mathcal{W}_q$; the choices for $W_q$
correspond to the strings $n_1,\ldots , n_v,c$
of non-negative integers such that
$e = d(n_1 + \cdots + n_v)+c$. We do this 
for each prime $q$ dividing $b$, and
get all $\langle s \rangle$-submodules as
 direct products of these parts. The module 
generating sets are sufficient to assemble group 
generating sets of the $T$-extensions in $\Mo(p,\C)$.

\subsection{Modules for every solvable permutation part}

Let $a\geq 1$ be a proper divisor of $p-1$. 
The next two lemmas enable us to refine the list 
$\mathcal A$ of Theorem~\ref{AllModulesEverywhere}
to a list of finite $\langle s, t^a\rangle$-modules,
and are also used in solving the conjugacy 
problem.
 
Note that each finite 
$\langle s\rangle$-submodule of $X_{\{p\}}$ is a 
$\langle t\rangle$-module, by Lemma~\ref{XModules}.
\begin{lemma}
\label{tConjugacyonspModules}
$Y_{j,k,l}^t = Y_{j,k,l'}$ where $l'$ is the image of 
$l$ under $t^j\in \mathrm{Sym}(p)$, i.e., 
$l'\equiv lu^j \bmod p$.
\end{lemma}
\begin{proof}
Let $m$ be the residue of $-j$ modulo $p-1$, and let
$n= (j+m)/(p-1)$. If $a_{p^k}$ denotes the diagonal matrix 
with $e^{2\pi \mathrm{i}/p^k}$ in position $(1,1)$ and $1$s
elsewhere on the main diagonal, then
$x_{p^{j+1}} = \gamma^m(a_{p^n})$ for $m\neq 0$ and
 $x_{p^{j+1}} = \gamma^{p-1}(a_{p^{n+1}})$ for $m= 0$.

Suppose that $m>\allowbreak 0$; the  
proof for $m=\allowbreak 0$ is similar. 
It may be checked that 
 $x_{p^{j+1}}^t = a_{p^n}^{t(1-s^u)^m}$  
and $a_{p^n}^t = a_{p^n}^{s^{u-1}}$.
Since $\langle s\rangle$ acts trivially 
on $X_{p^{j+1}}/X_{p^j}$,
\[
x_{p^{j+1}}^t \equiv a_{p^n}^{(1-s^u)^m} 
 \bmod X_{p^j}.
\]
Binomial expansion in $\Z\langle s\rangle$ 
gives 
\[
(1-s^u)^m = (1-s)^m u^m + 
\textrm{(terms divisible by }(1-s)^{m+1}).
\]
Thus $x_{p^{j+1}}^t \equiv  
 x_{p^{j+1}}^{u^m} \bmod X_{p^j}$.
 
Let $y=x_{p^{j+1}}z_{p^{k+1}}^l$. 
Since $Y_{j,k,0}$ is a $\langle t\rangle$-module, 
we assume that $l>0$. By the above,
$y^t \in \allowbreak 
 x_{p^{j+1}}^{u^m} z_{p^{k+1}}^l X_{p^j}$; so
$(y^t)^{u^j} \in 
 x_{p^{j+1}}z_{p^{k+1}}^{lu^j}X_{p^j}$.
Hence $Y_{j,k,l'}\subseteq Y_{j,k,l}^t$. 
As these modules have the same order, 
 they are equal.
\end{proof}

\begin{lemma}
\label{tConjugacyonsqModules} 
$(X^{(1)}_{q,n_1} X^{(2)}_{q,n_2} \cdots X^{(v)}_{q,n_v})^t =
 X^{(1)}_{q,n_v} X^{(2)}_{q,n_1} \cdots X^{(v)}_{q,n_{v-1}}$.
\end{lemma}
\begin{proof}
The only indecomposable direct factors of $X_{\{q\}}$ are 
the $X_{\{q\}}^{(r)}$. Thus $(X_{\{q\}}^{(r)})^t = 
X_{\{q\}}^{(j)}$ for some $j$. The $X_{q,1}^{(1)}, 
\ldots , X_{q,1}^{(v)}$ 
are pairwise non-isomorphic $\langle s\rangle$-modules, so it 
is enough to prove that $(X_{q,1}^{(r)})^t \cong X_{q,1}^{(r+1)}$ 
(reading superscripts modulo $v$).

By Definition~\ref{Definefandgr},
$g_{r+1}({\rm x})$ divides $g_r({\rm x}^u)$, whence 
$g_r(s)^t \in  g_{r+1}(s)\Z_q\langle s\rangle$. 
Each element of $X_{q,1}^{(r+1)}$
is annihilated by $g_{r+1}(s)$, and therefore by 
$g_r(s)^t$. Thus $(X_{q,1}^{(r+1)})^{t^{-1}} 
\subseteq X_{q,1}^{(r)}$; then
$(X_{q,1}^{(r)})^{t} = X_{q,1}^{(r+1)}$ by 
Lemma~\ref{411Bacskai}.
\end{proof}

\section{Monomial groups with cyclic permutation part}
\label{MonomialGroupsWithCp}

This section presents our first solutions of the extension 
and conjugacy problems. The resulting classification 
subsumes that in \cite[\S\S\! 2--3]{Conlon}.
\begin{definition}
Let $\mathcal L$ be the set of all groups
$\langle sz_{p^{k+1}}^i, A\rangle$
where $0\leq i\leq p-1$ and $A\in \mathcal A$ as
in Theorem~\ref{AllModulesEverywhere} with 
$A\cap Z_{\{ p\}} = Z_{p^k}$.
\end{definition}

\begin{proposition}
\label{InitialPare}
A finite subgroup of $\widetilde{\Mo}(p,\C)$ with 
permutation part $\langle s\rangle$ is 
$D$-conjugate to a group in $\mathcal L$.
\end{proposition}
\begin{proof}
Let $G$ be a subgroup of $\widetilde{\mathrm M}(p,\C)$
with $\phi(G)=\langle s\rangle$, so $sxz\in G$ for some
torsion elements $x\in X$ and $z\in Z$. 
Since $\gamma(X) = X$, there exists $y\in X$ such that 
$G^y = \langle sz,D\cap G\rangle$.
Then $z^p=(sz)^p \in (G\cap Z_{\{p\}})(G\cap  Z_{\{p\}'})$
implies that we may multiply $z$ by scalars from $D\cap G$ 
to get $z\in Z_{\{p\}}$.
\end{proof}

By Theorem~\ref{NonScalarIffIrreducible},  
the irreducible groups in $\mathcal L$ are 
precisely the non-abelian ones: those with 
non-scalar diagonal subgroup.

Since each group in $\mathcal L$ is normalized 
by $\langle s\rangle$, and 
$N_{\mathrm{Sym}(p)}(\langle s\rangle) = 
\langle s, t\rangle$,
Theorems~\ref{GLConjugateImpliesSamePP} and 
 \ref{Protoconjugacy} 
guarantee that the irreducible groups with a 
 unique abelian normal subgroup of index $p$ 
 are $\GL(p,\C)$-conjugate if and only if 
they are $D\langle t\rangle$-conjugate. 
We decide conjugacy of this type using 
 the next two lemmas.
\begin{lemma}
\label{Littlet}
Let $G \in \mathcal L$, with $G\cap Z_{\{p\}} =
Z_{p^k}$. Then $G$ is $\langle t\rangle$-conjugate 
to some $H\in \mathcal L$ such that 
$s\in \allowbreak H$ or 
$s z_{p^{k+1}} \in H$.
\end{lemma}
\begin{proof}
If $i\neq 0$ and $i^{-1}\equiv u^e \bmod p$, for $u$ as
in Notation~\ref{stuNotation}, 
then $(sz_{p^{k+1}})^{t^e} \equiv (sz_{p^{k+1}}^i)^{u^e} 
\bmod Z_{p^k}$.
\end{proof}
\begin{lemma}
\label{ElGt0}
Distinct $G$, $H\in \mathcal L$ are $D$-conjugate if and only 
if $G\cap D = \allowbreak H\cap D$ and 
$G\cap D_{\{p\}} \in \{  1,  Y_{j,k,l} \; |\;  l\neq 0\}$.
\end{lemma}
\begin{proof}
Suppose that $G^d=H$ for some $d\in D$. Thus 
$G\cap H = D\cap H$. If
$G\cap D_{\{p\}}= Y_{j,k,0}\neq 1$,
then $[s,d]\in \allowbreak 
(G\cap D)Z \cap X \leq G\cap X\leq H$.
As a consequence, $H= G$.

If $G\cap D_{\{p\}} =1$ then all 
$\langle sz_p^i, G\cap D \rangle$ in $\mathcal L$ 
are $\langle x_{p^2}\rangle$-conjugate.

Lastly, suppose that $G\cap D_{\{p\}} = Y_{j,k,l}$
where $1\leq l \leq p-1$.
Now $\langle s, \allowbreak
 G\cap D_{\{ p\}}\rangle^x =
\langle sz_{p^{k+1}}^{-l}, 
G\cap D_{\{ p\}}\rangle$ for $x\in D$ such that 
$\gamma (x) = \allowbreak x_{p^{j+1}}$. 
Hence all $\langle sz_{p^{k+1}}^i, G\cap D  \rangle
\in \mathcal L$ are $\langle x \rangle$-conjugate.
\end{proof}

If $G, H\in \mathcal L$  
are conjugate by a non-monomial matrix, then
each has more than one abelian normal
subgroup of index $p$. 
Such a group $G$ has a scalar subgroup of index 
 $p^2$, so is nilpotent of class $2$. 
\begin{lemma}
\label{Vanderlike}\
\begin{itemize} 
\item[{\rm (i)}] 
The groups in $\mathcal L$ that are nilpotent of class
$2$ are the  
$\langle sz_{p^{k+1}}^i, Y_{1,k,l}, 
\allowbreak Z_m\rangle$ where $k \geq 1$, $l \geq 0$,
and $\mathrm{gcd}(p,m)=1$.
\item[{\rm (ii)}] Let $\mathrm{gcd}(p,m)=1$.
Up to $\GL(p,\C)$-conjugacy, there are exactly 
two groups of order $p^{k+2}m$ in $\mathcal L$ that 
are nilpotent of class $2$, namely 
$\langle s, Y_{1,k,0},Z_m\rangle$ and 
$\langle sz_{p^{k+1}}, Y_{1,k,0},Z_m\rangle$.
\end{itemize}
\end{lemma}
\begin{proof}
By Remark~\ref{YGeneratorsAndOrder}~(ii), if 
$G\in \mathcal L$ is nilpotent of class $2$ then
  $G\cap D_{\{p\}}=Y_{1,k,l}$ for some $l$ and
$k\geq 1$.

We prove (ii) for $p\geq 3$. By (the proof of) 
Lemma~\ref{Littlet},  
$H= \langle sz_{p^{k+1}}, \allowbreak Y_{1,k,0} \rangle 
   = \langle sz_{p^{k+1}}, \allowbreak x_{p^2} \rangle$
is conjugate to each group
$\langle sz_{p^{k+1}}^i, Y_{1,k,0}\rangle$ for
 $1<i\leq p- 1$.
Let $e\in \GL(p,\C)$ be the Vandermonde matrix 
with entry $\epsilon^{rc}$ in row $r$, column $c$, 
where $|\epsilon|=p$
(cf.~\cite[4.1]{Conlon}). Then $s^e = x_{p^2}$ and 
$x_{p^2}^e = s^{-1}$; so  $H^e = \langle s, Y_{1,k,1}\rangle$. 
Lemmas~\ref{tConjugacyonspModules} and \ref{ElGt0} 
show that $H$ and the 
$\langle sz_{p^{k+1}}^i, Y_{1,k,l}\rangle$ for 
$i\geq 1$ or $l\geq 1$ are all conjugate to each other. 
However, $H$ is not conjugate
to $\langle s , Y_{1,k,0}\rangle$: this group 
has an elementary abelian subgroup of order 
$p^3$, while $H$ does not.
\end{proof}

We now define sublists $\mathcal{L}_1$, $\mathcal{L}_2$, 
 $\mathcal{L}_3$, $\mathcal{L}_4$ 
of groups $\langle sz_{p^{k+1}}^i, A\rangle
\in \mathcal L$, on the way to eliminating redundancy in 
$\mathcal L$.
\begin{definition}
Groups in $\mathcal L_1$ have 
$A = \langle Y_{1,k,0}, Z_m\rangle$
for $m$ coprime to $p$.
Groups in 
$\mathcal{L}_2\cup \mathcal{L}_3\cup \mathcal{L}_4$ have
 $A=\langle Y_{j,k,l}, \textstyle{\Pi}_q W_q \rangle$ 
for finitely many primes $q\neq p$ where 
$W_{q} \in \mathcal{W}_{q}$. 
Conditions that govern membership of such
$\langle sz_{p^{k+1}}^i, A\rangle$ in
an $\mathcal L_i$ are as follows.
\begin{itemize}
\item[] \hspace{-10pt} 
 $\mathcal{L}_1$:  
 $i\in \{0, 1\}$ and $k\geq 1$.
\item[] \hspace{-10pt}
 $\mathcal{L}_2$: $i=1$; $l=0$; either $j=k=0$ or 
 $k\geq 1$; and either $j\geq 2$ 
or some $W_q$ is non-scalar.
\item[] \hspace{-10pt}
 $\mathcal{L}_3$: $i= l=0$; $k\geq 1$; and either $j\geq 2$ 
or some $W_q$ is non-scalar.
\item[] \hspace{-10pt}
 $\mathcal{L}_4$: $i=0$; $1\leq l\leq p-1$; $j, k\geq 1$;
and either $j\geq 2$ 
or some $W_q$ is non-scalar. 
\end{itemize}

\noindent Let $\mathcal{L}_0= 
\mathcal{L}_1 \cup \mathcal{L}_2 \cup\mathcal{L}_3 
\cup\mathcal{L}_4$.
\end{definition}

Each group in $\mathcal{L}_0$ 
has non-scalar diagonal subgroup, 
 hence is irreducible. 
\begin{theorem}
\label{CorrectPList}\
\begin{itemize}
\item[{\rm (i)}]
Up to $\GL(p,\C)$-conjugacy,
$\mathcal{L}_0$ is a complete list of the finite irreducible
subgroups of $\Mo(p,\C)$ with permutation part $C_p$. 
That is, an irreducible group in $\mathcal L$ is conjugate to 
at least one group in  $\mathcal{L}_0$. 
\item[{\rm (ii)}]  
Distinct $G, H \in \mathcal{L}_0$ are 
conjugate only if they are both in the same sublist 
$\mathcal{L}_2$, $\mathcal{L}_3$, or $\mathcal{L}_4$. 
If $G, H\in \mathcal{L}_2$ are conjugate, then 
$G\cap D_{\{p\}} = H\cap D_{\{p\}} = 1$ and 
$G\cap D$ is $\langle t\rangle$-conjugate
to $H\cap D$. 
If $G, H\in \mathcal{L}_i$ for $i=3$ or $4$ are 
conjugate, 
then $G$ is $\langle t \rangle$-conjugate to $H$.
\end{itemize}
\end{theorem}
\begin{proof}
By Lemma~\ref{Vanderlike}, every group in $\mathcal{L}_1$ 
is nilpotent of class $2$; a nilpotent group of
class $2$ in $\mathcal L$ is conjugate to a 
single group in $\mathcal{L}_1$;  and no group in 
$\mathcal{L}_1$ can be conjugate to a group in 
$\mathcal{L}_0\setminus \mathcal{L}_1$.

Assume now that $G\in \mathcal{L}$ is irreducible 
and not nilpotent of class $2$.
Then $G\cap D$ is the unique abelian normal subgroup of 
$G$ with index $p$, so $G$ is $\GL(p,\C)$-conjugate 
to $H\in \mathcal L$ only if $G$ is 
$D\langle t\rangle$-conjugate to $H$. 
A laborious check against the definition of the 
$\mathcal L_i$ and 
Lemmas~\ref{Littlet}--\ref{Vanderlike}
confirms that $G$ is conjugate to a group in 
$\mathcal{L}_2\cup \mathcal{L}_3\cup \mathcal{L}_4$. 

It remains to prove (ii) for
$G, H\in \mathcal{L}_2\cup \mathcal{L}_3\cup \mathcal{L}_4$.
Suppose that 
$G= \allowbreak \langle sz_{p^{k+1}}, Y_{j,k,0}, 
\allowbreak M\rangle \in \mathcal{L}_2$ and 
$H\in \mathcal{L}_3\cup\mathcal{L}_4$ 
are conjugate. 
Then $|G\cap D_{\{p\}}|=  |H\cap D_{\{p\}}|$ 
implies that $k\geq 1$. Since $H$ splits over its diagonal
subgroup, $G$ does too. But this is false:
there is no $d\in D\cap G$ such that 
$|sz_{p^{k+1}}d|=p$ for $k\geq 1$ (if there
were such a $d$, then $z_{p^{k}}= \allowbreak z_{p^{k+1}}^p$ 
would be in $\chi(Y_{j,k,0})=Z_{p^{k-1}}$).
 
A group $G\in \mathcal{L}_2$ has 
$p-1$ different $\langle t\rangle$-conjugates 
of the form $\langle sz_{p^{k+1}}^i, A\rangle$, 
one for each $i\in \{ 1, \ldots , p-1\}$.
By Lemmas~\ref{tConjugacyonspModules} and \ref{ElGt0}, 
if $G\cap D_{\{p\}} \neq 1$ then
the only one of these that is $D$-conjugate  
to a group in $\mathcal{L}_2$ is $G$ itself.
 
Since $\langle t\rangle$-conjugacy leaves  
$\mathcal L_3$ and $\mathcal L_4$ setwise
invariant, no group in $\mathcal L_3$ is 
$D\langle t\rangle$-conjugate 
to a group in $\mathcal L_4$ by Lemma~\ref{ElGt0}. 
If $G,H \in \mathcal{L}_3 \cup \mathcal{L}_4$ are 
$D$-conjugate then $G=H$; 
hence $G$ and $H$ can be conjugate only if they
are $\langle t\rangle$-conjugate.
\end{proof}
\begin{corollary}
\label{ConjugacyDeterminedByDiagonals}\
\begin{itemize}
\item[{\rm (i)}] Groups in $\mathcal{L}_2$ are 
conjugate if and only if their diagonal subgroups 
have $p'$-order and are $\langle t\rangle$-conjugate.
\item[{\rm (ii)}]
Groups in either $\mathcal{L}_3$ or $\mathcal{L}_4$ 
are conjugate if and only if their diagonal subgroups 
are $\langle t\rangle$-conjugate.
\end{itemize}
\end{corollary}

\begin{theorem}\label{p2Done}
If $p=2$ then $\mathcal{L}_0$ is a 
(complete and irredundant) 
classification of the finite irreducible subgroups 
of $\Mo(2,\C)$.
\end{theorem}
\begin{proof}
Here $t=1$. Since $\mathcal L_1$ is 
irredundant and there is a single group in 
$\mathcal{L}_0\setminus \mathcal{L}_1$ 
with given diagonal subgroup, the result 
follows from 
Theorem~\ref{CorrectPList} and 
Corollary~\ref{ConjugacyDeterminedByDiagonals}.
\end{proof}

If $G\in \mathcal{L}_2\cup \mathcal{L}_3$ 
then we need only worry about $\langle t\rangle$-orbits 
in the $\mathcal{W}_q$. However, if $G\in \mathcal{L}_4$ 
then $\langle t\rangle$-conjugacy might change
$G\cap D_{\{p\}}$.

We encode the action of $\langle t\rangle$ on 
the set of finite $\langle s\rangle$-submodules
of $D_{\{ p\}'}$ as an action by $\langle t \rangle$ 
on a set of arrays ${\sf N}$.
Each such array has $p-1$ columns and 
finitely many non-zero rows. 
Let $q\neq p$, and suppose that the Sylow
$q$-subgroup of $A\in \mathcal A$ is 
$X^{(1)}_{q,n_1} X^{(2)}_{q,n_2} \cdots \allowbreak 
X^{(v)}_{q,n_v}Z_{q^c}$, where as usual $(p-1)/v$ 
 is the multiplicative order of $q$ modulo $p$.
Row $q$ of the array ${\sf N}_A$ has $n_r$ in column $r$ 
for $r\leq v$, and  $n_{r-v}$ in  column $r$ for 
$r> v$. 
By Lemma~\ref{tConjugacyonsqModules}, 
${\sf N}_A^t$ is the array obtained from 
${\sf N}_A$ by shifting columns 
of ${\sf N}_A$ one place rightward, modulo $p-1$.

A lexicographic ordering is defined on these arrays.
Specifically, ${\sf N}\leq {\sf N}'$
 if and only if the first entry in the first 
row of ${\sf N}$ where  ${\sf N}$ and ${\sf N}'$ 
differ is at most the matching entry 
in ${\sf N}'$.
We select a minimal element of each 
$\langle t\rangle$-orbit of ${\sf N}_A$. 
Although cumbersome, 
this formulation of $\langle t \rangle$-conjugacy 
in $\cup_q\mathcal{W}_q$ is easily automated. 

\begin{definition}\label{StarLDefinitions}
Let $\mathcal{L}^*:= \mathcal{L}_1\cup \mathcal{L}_2^*
 \cup \mathcal{L}_3^* \cup \mathcal{L}_4^*$, where
\begin{itemize}
\item[(i)] 
$\mathcal L_2^*$ consists of those 
$G \in \mathcal{L}_2$ 
such that either 
$G\cap D_{\{p\}}\neq 1$, or ${\sf N}_{G\cap D}$ is 
$\langle t\rangle$-minimal (i.e., minimal
in its $\langle t\rangle$-orbit); 
\item[(ii)]
$\mathcal L_3^*$ consists of those $G\in \mathcal L_3$ 
such that ${\sf N}_{G\cap D}$ is 
$\langle t\rangle$-minimal;
\item[(iii)]
 $\mathcal L_4^*$ consists of
those $G\in \mathcal L_4$ such that
\begin{itemize}
\item
$G\cap D_{\{p\}} = Y_{j,k,l}$ for
$l\in \{u,u^2,\ldots , \allowbreak u^{j'}\}$ modulo $p$
where $j'=\mathrm{gcd}(j,p-1)$ and
$u$ is the least primitive integer 
modulo $p$,
\item 
${\sf N}_{G\cap D}$ is  
$\langle t^{(p-1)/j'}\rangle$-minimal.
\end{itemize}
\end{itemize}
\end{definition}

Theorem~\ref{CorrectPList}, 
Corollary~\ref{ConjugacyDeterminedByDiagonals}, 
and the foregoing provide our first major classification 
of irreducible monomial groups.
\begin{theorem}\label{sGroupsFinal}
Up to $\GL(p,\C)$-conjugacy, 
$\mathcal{L}^*$ is a complete and
irredundant list of the finite irreducible subgroups of 
$\mathrm{M}(p,\C)$ with permutation part $C_p$.
\end{theorem}
\begin{proof}
We prove that $\mathcal{L}^* \setminus \mathcal{L}_1$ 
is irredundant and complete.

By Definition~\ref{StarLDefinitions} and 
Corollary~\ref{ConjugacyDeterminedByDiagonals}, 
each group in $\mathcal L_3$ is conjugate to one 
in $\mathcal{L}_3^*$.
Suppose that $G, H\in \mathcal{L}_3^*$ 
are conjugate. 
Minimality and Corollary~\ref{ConjugacyDeterminedByDiagonals} 
force $G\cap D =\allowbreak  H\cap D$. 
But the groups in $\mathcal{L}_3$ are distinguished 
by their diagonal subgroups, so $G = H$.
The reasoning for $\mathcal{L}_2^*$ is similar.

Now let $G=\langle s, G\cap D\rangle \in \mathcal{L}_4$,
 with $G\cap D_{\{p\}} = Y_{j,k,l}$. 
By Lemma~\ref{tConjugacyonspModules},
there is a unique non-negative integer $a$ such that 
$a<(p-1)/j'$ and $H\cap D_{\{ p\}} = Y_{j,k,\ell}$
 where $H=G^{t^a}$ and 
$\ell \in \{ u, u^2, \ldots , u^{j'}\}$ modulo $p$.
Conjugation of $H$ by some $t^b$ 
preserves this value of $\ell$ (i.e., does not change
$\langle s, H\cap D_{\{p\}}\rangle$) if and 
only if $b$ is divisible by $(p-1)/j'$.
Completeness of $\mathcal L^*$ is proved.

If $G, H\in \mathcal{L}_4^*$ are conjugate then 
$G\cap D_{\{p\}}$ and $H\cap D_{\{p\}}$
are $\langle t\rangle$-conjugate. 
By the uniqueness statement above,
$G\cap D$ and $H\cap D$ can only be
$\langle t^{(p-1)/j'}\rangle$-conjugate; so they 
are the same. Hence $G=H$.
\end{proof}

\section{The remaining solvable monomial groups}

In this section we classify the finite irreducible 
solvable subgroups 
of $\Mo(p,\C)$ with non-cyclic permutation part. 
To that end, $p$ is assumed odd (by 
Theorem~\ref{p2Done}).

\begin{definition}\label{NoncyclicSolvableTransitive}
Let $T = \langle s , t^a\rangle$ 
where $a\geq 1$ is a proper divisor of $p-1$, and 
let $\hat{a}=(p-1)/a$.
\end{definition}
Up to conjugacy, the groups $T\cong C_p\rtimes C_{\hat{a}}$ 
in Definition~\ref{NoncyclicSolvableTransitive} 
are the non-cyclic solvable transitive subgroups 
of $\mathrm{Sym}(p)$. 

Recall the discussion before
Definition~\ref{StarLDefinitions} of 
$\langle t\rangle$-conjugacy in $\cup_q \mathcal W_q$.
\begin{lemma}\label{taSubmodules}
The subset $\mathcal{A}^{[a]}$ of $\mathcal{A}$ 
(see Theorem{\em ~\ref{AllModulesEverywhere}}) consisting of 
all $A$ such that ${\sf N}_A^{t^a} = \allowbreak {\sf N}_A$
and either $l=0$ or $aj\equiv 0 \bmod p-1$, where 
$A\cap D_{\{p\}} = Y_{j,k,l}$,
is the set of all finite $T$-submodules of $D$. 
\end{lemma}
\begin{proof}
We refine $\mathcal{A}$ using 
Lemma~\ref{tConjugacyonspModules}.
\end{proof}

If $A\in \mathcal{A}^{[a]}$ and 
$X^{(1)}_{q,n_1} \cdots X^{(v)}_{q,n_v}$ is the
Sylow $q$-subgroup of $A\cap X$, then 
${\sf N}_A^{t^a} = {\sf N}_A$ is equivalent to 
$n_r = n_{r+a}$ for all $r$ and $q$
(see Lemma~\ref{tConjugacyonsqModules}). These
conditions are again straightforward to implement, 
building on our implementation of $\mathcal A$.

\begin{definition}\label{calMaDefinition}
Let $\mathcal{M}^{[a]}$ be the set of all
$\langle s, t^a z_{m\hat{a}}^{c}, A\rangle$ where 
$A\in \mathcal{A}^{[a]}$,
$A\cap Z_{\{p\}'}=\allowbreak Z_m$, 
and $0\leq c < \hat{a}$.
\end{definition}

\begin{theorem}\label{sSplitsOff}
A finite subgroup of $DT$ with permutation
part $T$ is $D$-conjugate to a group in 
$\mathcal{M}^{[a]}$.
\end{theorem}
\begin{proof}
Denote $t^a$ by $\thickbar{t}$.
Let $G$ be a finite subgroup of 
$DT$ such that $\phi(G) = T$.
Put $A=G\cap D$ and 
$F=\phi^{-1}(\langle s\rangle)\cap G$.
By Proposition~\ref{InitialPare},
we assume that $F\in \mathcal L$.

The Frattini argument shows that if $P$ 
is a Sylow $p$-subgroup of $F$, then there 
is $h\in N_G(P)$ such that 
$\phi(h)=\thickbar{t}$. We may replace 
$h$ by an appropriate power of $h$ to arrange 
that $h^{\hat{a}}\in \allowbreak A_{\{p\}'}$.
Choose $g\in P$ with $\phi(g)=s$. Then 
$[h^{\hat{a}},g ]\in  P\cap A_{\{p\}'}=1$. 
Hence $h^{\hat{a}}$ is scalar (it is
 centralized by $s$). 
 
So $h^{\hat{a}}= \lambda^{\hat{a}} 1_p$ 
for $\lambda\in \C^\times$ of $p'$-order 
such that $\mathrm{det}(h)= \lambda^p$.
Define $y = (-1)^a\lambda 1_p$ and  
$x=\thickbar{t}^{\, -1}hy^{-1}$. 
Then $y^{\hat{a}} = h^{\hat{a}}$,
$x\in X$ (as $\mathrm{det}(x)=1$), 
and $h^{\hat{a}}=x^\psi y^{\hat{a}}$ 
where $\psi$ is the element
$1+\thickbar{t}+\thickbar{t}^{2} + 
\cdots \allowbreak + 
\thickbar{t}^{\, \hat{a}-1}$ 
of the integral group ring 
$\Z\langle s,t\rangle$. 
Thus $x^\psi=1$.
 
Now $g=sz$ for $z\in Z_{\{p\}}$ such that 
$z^p\in A$.
Since $g^h \in F$ and $s^t=s^u$, 
we have $d:= \allowbreak x^{1-s^{\thickbar{t}}} 
 \in \allowbreak z^{u^a-1}A$ (remember $s^t=s^u$). 
Raising $d$ to the power $1+s^{\thickbar{t}} + 
\cdots \allowbreak +s^{\thickbar{t}(\tilde{u}-1)}$,
where $\tilde{u}\equiv u^{-a} \bmod p$,
reveals that $x^{1-s} \in z^{1-\tilde{u}}A$.
Thus $\gamma^{p-1}(x) \in A$.
The $\langle s\rangle$-module generated by $x^p$ 
is the same as the one generated by $\gamma^{p-1}(x)$, 
so $x^p\in A$.

Fortunately, $x$ can be conjugated away. 
Let $\mu =\thickbar{t}+2\thickbar{t}^2+\cdots + 
(\hat{a}-1)\thickbar{t}^{\, \hat{a}-1}$. 
Then $\mu(1-\thickbar{t})=\psi-\hat{a}$, and we
further calculate that $G$ is conjugate by 
$x^{-a\mu}$ to
\[
\langle szx^{a\mu(s-1)}\! , \hspace{.1pt} 
\thickbar{t}y, \hspace{.1pt} \allowbreak A\rangle.
\]
(Note: the inclusion 
 $(\hspace{1pt} \thickbar{t}y)^{\hat{a}}\in A \cap Z_{\{p\}'}$
implies the possibilities for the generator 
with permutation part $\thickbar{t}$   
in Definition~\ref{calMaDefinition}.)
Starting from the identity
\[
t^b(s-1) = (s-1)(1+s+ \cdots 
+ s^{u^{-b}-1})t^b
\]
in $\Z\langle s, t \rangle$,
we can prove the existence of
 $\nu\in \Z\langle s, t \rangle$
with coefficient sum 
\[
c \equiv \tilde{u}+2\tilde{u}^{2}+\cdots + 
(\hat{a}-1) \tilde{u}^{(\hat{a}-1)}\bmod p
\]
such that $\mu(s-1)=(s-1)\nu$. Together with 
 $x^{s-1} \in z^{\tilde{u}-1}A$, this yields
 $z^{(1-\tilde{u})a c}x^{a\mu(s-1)}\in A$.
Also,
\begin{align*}
(\tilde{u}-1)c 
& \equiv -\tilde{u}-\tilde{u}^2 - \cdots - 
\tilde{u}^{\hat{a}-1} 
+ (\hat{a}-1)\tilde{u}^{\hat{a}}\\
& \equiv 1+(\hat{a} -1)\tilde{u}^{\hat{a}}\\
& \equiv \hat{a} \bmod p.
\end{align*}
Thus $(\tilde{u}-1)ac \equiv -1 \bmod p$.
It follows that 
$zx^{a\mu(s-1)}\in A$, as required.
\end{proof}
\begin{remark}
 The relative simplicity of the family $\mathcal{M}^{[a]}$ 
places restrictions on degree $p$ representations of 
$\mathrm{Aut}(G)$ for 
$G\in \mathcal L$ (cf.~\cite[\S \! 6]{Conlon}).
\end{remark}

We move on to the conjugacy problem.
\begin{lemma}\label{MonomialConjugacyta}
Each group in $\mathcal{M}^{[a]}$ is normalized by 
$\langle s\rangle$, and stays in $\mathcal{M}^{[a]}$ 
under $\langle t\rangle$-conjugation. 
If $G^d \in \mathcal{M}^{[a]}$ 
for $d\in D$, 
then $G^d=G$. 
\end{lemma}

\begin{theorem}\label{teConjugacy}
If $G$, $H\in \mathcal{M}^{[a]}$ are irreducible 
and $\GL(p,\C)$-conjugate, then 
either they are $\langle t\rangle$-conjugate, or they are
$\langle e\rangle$-conjugate, where $e$ is 
 a Vandermonde matrix.
The latter can occur only when each of $G$ and $H$ has 
more than one abelian normal subgroup with quotient $T$.
\end{theorem}
\begin{proof}
Suppose that $G^w = H$ for some $w\in \GL(p,\C)$. 
If $(G\cap D)^w = H\cap D$ then 
$w\in D\langle s, t \rangle$ up to scalars.
Thus $G$ and $H$ are $\langle t\rangle$-conjugate
by Lemma~\ref{MonomialConjugacyta}.

Suppose now that $(G\cap D)^w \neq H\cap D$.
By the definition of $\mathcal Y$ and 
Lemma~\ref{taSubmodules},
$G\cap D = \langle x, G\cap Z \rangle$ where $x= x_{p^2}$; 
i.e., $G\cap D_{\{p\}} = 
H\cap D_{\{p\}} = Y_{1,k,0}$ for some $k$. 
The Fitting subgroup $\langle s,G\cap D\rangle$ of $G$ 
(and of $H$) has unique Sylow $p$-subgroup 
$E=\langle s,x\rangle$, an extraspecial group 
of order $p^3$ and exponent $p$. Therefore 
$w\in N:=N_{\GL(p,\C)}(E)$.

It is proved in \cite[\S \! 3]{BoltRoomWall} that 
$N_{\PGL(p,\C)}(EZ/Z)/(EZ/Z) \cong \SL(2,p)$.
As the largest subgroup
of exponent $p$, $E$ is characteristic in $EZ$.
Hence $N/EZ \cong \SL(2,p)$.

For the Vandermonde matrix $e$ defined in the proof 
of Lemma~\ref{Vanderlike},
 $s^e = x$ and $x^e = s^{-1}$. Also $t^e = t^{-1}$.
(We know from the proof of 
Lemma~\ref{tConjugacyonspModules} that $x^t = x^{\tilde{u}}$ 
where $\tilde{u}\equiv u^{-1} \bmod p$. 
Thus $t^et$ centralizes $E$. So the scalar
$t^et$ is $1$, because $t^{-1}$ and $t^e$  
have trace $1$.)
Let $\sigma$ be the natural surjection of 
$N$ onto $\SL(2,p)$.  
If $d= x_{p^3}$ then $\sigma(d)$ has order $p$, so 
generates a Sylow $p$-subgroup of $\SL(2,p)$.
Its normalizer $\sigma(\langle d,t\rangle)$ is 
maximal in $\SL(2,p)$.
Consequently $\SL(2,p)=\sigma(\langle d,t,e\rangle)$.

Observe that $\sigma(w)$ normalizes 
$\langle \sigma(t^a) \rangle=\sigma(G)=\sigma(H)$.
Representing $\sigma(d)$, 
$\sigma(t)$, and $\sigma(e)$ in $\SL(2,p)$ 
according to conjugation action on the basis
$\{sZ(E), \allowbreak xZ(E)\}$ of $\mathbb{F}_p^2$, 
we find that $\sigma(\langle t,e\rangle)$ is 
monomial. 
If the diagonal subgroup $\sigma(\langle t^a \rangle)$ 
is non-scalar then it has normalizer 
$\sigma(\langle t,e\rangle)$ in $\SL(2,p)$;
otherwise $G^d = G$.
 Since $G^t = G$,
this completes the proof.
\end{proof}

\begin{definition}
Let $\mathcal{M}_1^{[a]}$,
 $\mathcal{M}_2^{[a]}$, and
 $\mathcal{M}_3^{[a]}$ respectively denote the sublists 
of $\mathcal{M}^{[a]}$ 
consisting of $G= \langle s, t^a z_{m\hat{a}}^{c}, 
Y_{j,k,l}, G\cap D_{\{p\}'} \rangle$ that satisfy 
(1), (2), (3) below.
\begin{enumerate}
\item $l=0$, $0\leq c \leq \hat{a}/2$, $j=1$, 
and $\mathsf{N}_{G\cap D}=0$.
\item[] \vspace{-15pt}
\item 
$l=0$, either $j\geq 2$ or $\mathsf{N}_{G\cap D}\neq 0$, and 
$\mathsf{N}_{G\cap D}$ is $\langle t\rangle$-minimal.
\item[] \vspace{-15pt}
\item  $j\geq 2$ is divisible by $\hat{a}$, $l\in
\{u, u^2, \ldots , u^{j'}\}$ modulo $p$ where 
$j'=\mathrm{gcd}(j,p-1)$, and $\mathsf{N}_{G\cap D}$ is 
$\langle t^{(p-1)/j'}\rangle$-minimal. 
\end{enumerate}
Let $\mathcal{M}^*$ be the union of all 
$\mathcal{M}^{[a]*} := \mathcal{M}_1^{[a]}\cup \mathcal{M}_2^{[a]}
\cup \mathcal{M}_3^{[a]}$ as $a$ ranges over the proper divisors 
of $p-1$. 
\end{definition}
 
The solvable groups are now classified. 
\begin{theorem}\label{SolvableFullMonp}
For $\mathcal{L}^*$ as in 
Definition{\em ~\ref{StarLDefinitions}}, 
$\mathcal{L}^*\cup \mathcal{M}^*$ 
is a complete and irredundant list of the 
solvable finite irreducible monomial subgroups 
of $\GL(p,\C)$.
\end{theorem}
\begin{proof}
All groups in $\mathcal{M}^*$ are irreducible 
(Theorem~\ref{NonScalarIffIrreducible}). 
By Theorems~\ref{GLConjugateImpliesSamePP}, 
\ref{sGroupsFinal}, and \ref{sSplitsOff}, 
we must show that each element of $\mathcal{M}^{[a]}$ 
(for fixed $a$) is conjugate to one and only one
 element of $\mathcal{M}^{[a]*}$.

Let $G=\langle s, t^az_{m\hat{a}}^c, A\rangle\in 
 \mathcal{M}^{[a]}$ where $A\cap D_{\{p\}} = Y_{j,k,l}$.
Heeding Lemma~\ref{taSubmodules}, we first suppose that $l = 0$. 
By Lemma~\ref{tConjugacyonspModules} and 
Theorem~\ref{teConjugacy},  
if $G$ has a unique abelian normal 
subgroup with quotient $T$, then the only
group in $\mathcal{M}^{[a]*}$ that is 
conjugate to $G$ lies in 
$\mathcal{M}^{[a]}_2$.  
Otherwise, $j=1$ and $\mathsf{N}_{G\cap D} = 0$. 
We have $G^e = \langle s, t^a z_{m\hat{a}}^{\hat{a}-c}, A\rangle$.
Thus the $\langle e\rangle$-orbit of $G$ 
contains just one group $H$ 
with $c\leq \hat{a}/2$. Since $H\in\mathcal{M}^{[a]}_1$, the 
only element of $\mathcal{M}^{[a]*}$ conjugate to $G$ is
$H$.

Suppose next that $l \geq 1$, so 
$j$ is positive and $j \equiv 0 \bmod \hat{a}$.
Only $\langle t\rangle$-conjugacy matters here, and 
$G$ cannot be conjugate to a group 
in $\mathcal{M}^{[a]}_1\cup \mathcal{M}^{[a]}_2$.
The rest of the proof echoes the last two paragraphs
in the proof of Theorem~\ref{sGroupsFinal}.
\end{proof}

\section{Non-solvable monomial groups}

Our objective in this section is to prove
general-purpose results 
for (\mbox{finite}) subgroups of $\Mo(p,\C)$
with non-solvable transitive permutation part $T$.
In later sections, we treat
$T=\Sym(p)$, $T=\Alt(p)$, and special cases of $T$
required to facilitate the classifications  
for $p\leq 11$.
\begin{notation}\label{InForce}
\emph{Let $p\geq 5$, let $q$ be a prime,
let $T$ be a non-solvable 
subgroup of $\mathrm{Sym}(p)$ containing 
$s$, and let $\pi$ 
 be the set of primes 
other than $p$ that divide~$|T|$.} 
\end{notation}
Note that all primes in $\pi$ are less than  $p$.

A finite $T$-submodule of $D_{\{p\}}$ is sandwiched 
between $\Omega_nX_{\{p\}}$ and 
$Z_{\{p\}} \cdot \Omega_nD_{\{p\}}$ for
some $n \geq 0$. 
\begin{lemma}\label{SpModulesp}
If $A=Y_{j,k,l}\in \mathcal Y$ then the following 
are equivalent.
\begin{itemize}
\item[{\rm (i)}] $A$ is a $T$-module. 
\item[{\rm (ii)}] $A$ is a $\mathrm{Sym}(p)$-module.
\item[{\rm (iii)}] 
Either $j\equiv 0 \bmod (p-1)$, or both
$l=0$ and $j\equiv -1 \bmod (p-1)$. 
\end{itemize}
\end{lemma}
\begin{proof} 
Let $j = n(p-1)$. Then $X_{p^j} = \Omega_n X_{\{p\}}$ 
and thus $Y_{j-1,k,0}= X_{p^j}Z_{p^k}$ are 
$\mathrm{Sym}(p)$-modules. Since $M:=\Omega_n D_{\{p\}}$ 
is a $\mathrm{Sym}(p)$-module of order $p^{np}$, 
and $Y_{j-1,n,0}$ has index $p$ in $M$, it follows that 
$M= Y_{j,n,l}$ for some $l\neq 0$. 
Hence $X\cap Z M = X_{p^{j+1}}$
is a $\mathrm{Sym}(p)$-module. 
Taking $p^n$-powers, we deduce that  
$X_{p^{j+1}}/X_{p^{j}}\cong Z_p$ is trivial 
as a $\mathrm{Sym}(p)$-module, so 
 (iii) $\Rightarrow$ (ii) by 
Theorem~\ref{GoursatRemak}.

Suppose that $A$ is a (non-identity) 
$T$-module. If $l\neq 0$  (resp., $l =0$)
then $A\cap X = X_{p^j}$ 
 (resp., $A\cap \allowbreak X = X_{p^{j+1}}$). 
By \cite[Satz~5.1]{Neumann2}, 
the only non-identity proper $T$-submodules of 
$\Omega_1 D_{\{p\}}$ are $X_{p}=Z_p$ and  
$\Omega_1 X_{\{p\}}= X_{p^{p-1}}$.
Since $(A\cap X)^r \in \Omega_1 X_{\{p\}}$ for
some $p$-power $r$, 
necessarily $A\cap X$ is $X_{p^{n(p-1)}}$ or 
$X_{p^{n(p-1)+1}}$ for some $n\geq 1$. 
The permitted values of $j$ and $l$ in (iii) are 
now evident.
\end{proof}

We derive a weaker 
statement for submodules of $p'$-order.
\begin{proposition}\label{XqIndecomposable}
If $q\neq p$ then
$X_{\{q\}}$ is an indecomposable $T$-module. 
\end{proposition}
\begin{proof}
Let $W = \Omega_1D_{\{q\}}$. We prove that the $\F_q$-space 
$\mathrm{End}_T W$ 
has dimension $2$. 
Since $W= Z_q\times \Omega_1 X_{\{q\}}$, 
this will imply that $\Omega_1 X_{\{q\}}$ 
and thus $X_{\{q\}}$ are indecomposable $T$-modules.

The permutation matrix group $T$ embeds in $\GL(p,q)$ under 
entrywise reduction modulo $q$, and 
$T$ thereby acts on $\mathrm{Mat}(p,\F_q)$ by conjugation,
with fixed-point space $\mathrm{End}_T W$.
By a result of Burnside~\cite[Theorem~3]{Neumann},
$T$ is $2$-transitive. 
Hence the elementary matrices in $\mathrm{Mat}(p,\F_q)$ 
are permuted in two orbits by $T$:
\[
\{ [\delta_{i,k}\delta_{l,j}]_{i,j} \mid \allowbreak
k\neq l, \, 1\leq k, l \leq p \} 
\qquad \mbox{and} \qquad
\{ [\delta_{i,k}\delta_{k,j}]_{i,j} \mid 
1\leq k \leq p\}.
\]
Summing the elements in each orbit gives 
 a basis of $\mathrm{End}_T W$.
\end{proof}

\begin{corollary}\label{qNotInpiModules}
For $q$ coprime to $|T|$, the $\Omega_n X_{\{ q\}}$ 
are all the finite $T$-submodules
of $X_{\{ q\}}$.
\end{corollary}
\begin{proof}
(Cf.~the proof of Lemma~\ref{XModules}.)
By Maschke's Theorem, 
$\Omega_1 X_{\{ q\}}$ is a 
completely reducible $\F_q T$-module.
Hence $\Omega_1 X_{\{ q\}}$ is irreducible by
Proposition~\ref{XqIndecomposable}, so is in every
non-identity $T$-submodule of $X_{\{ q\}}$. 
\end{proof}

We show that $p$ and primes not 
dividing $|T|$ can be set aside from
submodule orders that appear in our solution of the 
extension problem for $T$.
%: a welcome simplification.
\begin{proposition}\label{FirstNonSolvableReduction}
Let $G$ be a finite subgroup of 
$DT$ such that $\phi(G)=T$. 
Then there exists $H\leq TD_\pi$ such that $s\in H$
and $G$ is $D$-conjugate to  $H.(G\cap D_{\pi'})$.
\end{proposition}
\begin{proof}
By \cite[Satz~V.21.1~c)]{Huppert}, 
$N_{T}(\langle s\rangle)\neq \langle s\rangle$.
Thus $G$ has a subgroup with 
solvable non-cyclic permutation part, 
and we may suppose that $s\in G$ 
by Theorem~\ref{sSplitsOff}. 

Denote the natural surjection of $TD=GD$ onto 
 $TD/(G\cap D_{\pi'})D_\pi$ by an
overline; then
$\overline{GD}=\overline{TD} = 
\overline{G}\ltimes \overline{D} = 
\overline{T}\ltimes \overline{D}$.
Also, $|\overline{T}:\overline{T}\cap \overline{G}|$ 
is a $\pi$-number,
while $\overline{D}$ is a $\pi'$-group. 
Therefore, by 
\cite[Lemma~1, corrected]{DixonInfiniteComp},
 $\overline{T}$ and $\overline{G}$ are 
$\overline{D}$-conjugate. 
This implies that, for some $d\in D_{\pi'}$, 
\[
G^dD_{\pi} = (GD_\pi)^d = 
TD_{\pi}(G\cap D_{\pi'})=
TD_{\pi}(G^d\cap D_{\pi'}) .
\]  
Then $G^d =  H.( G\cap D_{\pi'})$ where 
$H= G^d\cap TD_{\pi}$.
Since $s[s,d]=s^d \in TD_{\pi}(G^d\cap D_{\pi'})$
and thus $[s,d]\in D_{\pi}(G^d\cap D_{\pi'})$,
 we have $[s,d]\in G^d\cap D_{\pi'}$.
Hence $s= s^d [s,d]^{-1} \in G^d\cap T \leq H$.
\end{proof}
 
So each group with non-solvable 
permutation part $T$ in our final list is 
the semidirect product of a $T$-submodule 
of $D_{\pi'}$ by a `hub group'  
in $TD_\pi$ containing $s$.

The next result is a companion piece to 
Proposition~\ref{FirstNonSolvableReduction},
dealing with a ubiquitous kind of
 hub group. If the hypotheses
are fulfilled, then we can discard even more of a
 diagonal subgroup when solving the 
extension problem. (Recall that for a  
group $K$ and $K$-module $U$, there is a
one-to-one correspondence between the first
cohomology group $H^1(K,U)$ and  
the set of conjugacy classes of complements of 
$U$ in $U\rtimes K$.)
\begin{proposition}\label{1stCoh}
Let $G$ be a finite subgroup of $TD_\pi$ such 
that $\phi(G)=T$ and $s\in G$. Suppose that 
$D_{\pi} = B\times C$ for $T$-modules
$B$ and $C$ where $H^1(T,B/G\cap B) = 0$. 
Then $G$ is $B$-conjugate to $H.(G\cap B)$ for 
some $H\leq TC$ such that $\phi(H)=\allowbreak T$
and $s\in H$.
\end{proposition}
\begin{proof}
By Theorem~\ref{qsSubmodules}, 
$G\cap D_\pi = (G\cap B)(G\cap C)$.
We mimic the proof of 
Proposition~\ref{FirstNonSolvableReduction}, with 
$D_\pi$, $B$, $C$ here in place of $D$, $D_{\pi'}$, 
$D_\pi$ there, respectively;  $\overline{G}$ and 
$\overline{T}$ are $B$-conjugate because 
$H^1(T, D_\pi/(G\cap B)C ) 
\cong H^1(T,B/G\cap B)  = \allowbreak 0$.
\end{proof}

Occasionally the $T$-module structure of $B/(B\cap G)$ 
is independent of $G$. 
\begin{lemma}\label{ZetaLem}
Let $\zeta$ be a subset of $\pi$ consisting of
$q$ such that $\Omega_1 X_{\{q\}}$ is an irreducible 
$T$-module, and let 
$B= \prod_{q\in \zeta} X_{\{q\}}\cdot  
\prod_{q\in \eta} Z_{\{q\}}$
where $\eta$ is any subset of $\pi$.
Then $B/A\cong B$ as $T$-modules for every 
finite $T$-submodule $A$ of $B$. 
\end{lemma}
\begin{proof}
It suffices to assume that $B = X_{\{q\}}$ 
for $q\in \zeta$.
Since $\Omega_1 B$ is an irreducible $T$-module, $A$ 
must be some $\Omega_n B$. 
Certainly $B/\Omega_n B\cong B$.
\end{proof}

We record basic results for calculating 
first cohomology.
These use the following definition.
If $R\leq T$ and $M$ is a right $R$-module,
then $M_R^T$ denotes the $T$-module 
co-induced from  $M$. 
That is, $M_R^T$ has 
element set $\mathrm{Hom}_R(\Z T,M)$, and becomes
a $T$-module by setting $\rho y(x) = \rho(yx)$ for 
$\rho \in \mathrm{Hom}_R(\Z T,M)$, $y\in T$,
and $x\in \Z T$.
\begin{lemma}[Eckmann--Shapiro~{\cite[p.~561]{Rotman}}]
\label{Shapiro} 
$H^n(R,M) \cong H^n(T,M_R^T)$ for all $n\geq 0$.
\end{lemma}
\begin{lemma}\label{Spminus1}
Let $R = T\cap \mathrm{S}_{p-1}$, where 
$\mathrm{S}_{p-1}\cong \mathrm{Sym}(p-1)$ 
is the group of permutation matrices in 
$\GL(p,\C)$ whose elements have $1$ in 
position $(p,p)$. Then 
$D_{\{q\}}\cong (Z_{\{ q\}})_R^T$ as $T$-modules.  
\end{lemma}
\begin{proof}
Since each $\rho \in (Z_{\{ q\}})_R^T$ is 
determined by its values on a transversal $U$ for 
the $p$ cosets of $R$ in $T$, as a group $(Z_{\{ q\}})_R^T$
is isomorphic to the group of all set maps $U\rightarrow
Z_{\{ q\}}$, which in turn is isomorphic to $D_{\{ q\}}$.

Let $\theta$ be the $R$-homomorphism from
$D_{\{ q\}}$ into $Z_{\{ q\}}$ defined by 
$\mathrm{diag}(a_1, \ldots , a_p) \mapsto \allowbreak
\mathrm{diag}(a_p,\ldots , a_p)$. 
By \cite[Theorem~4.9, p.~55]{HuppertBlackburnII},
there is a $T$-homomorphism $\theta'\colon D_{\{ q\}}\rightarrow
 (Z_{\{ q\}})_R^T$ with kernel in $\ker \theta$.
But $\ker \theta$ contains no non-identity $T$-modules.
Thus  $\theta'$ is an isomorphism, as desired.
\end{proof}
\begin{remark}
If $K$ is perfect and $M$ is a trivial
$K$-module, then $H^1(K,M) =0$. 
\end{remark}

\section{Permutation part $\mathrm{Sym}(p)$}
\label{PermPartSymp}

We maintain Notation~\ref{InForce}, writing 
$\mathrm{S}_n$ for $\mathrm{Sym}(n)$. 
\begin{definition}\label{MathcalAsDefn}
Let $\mathcal{A}^{[S]}$ be the set of 
$A\in \mathcal A$ such that ${\sf N}_A^t = {\sf N}_A$, 
and if $A\cap D_{\{p\}} = Y_{j,k,l}$, then either 
$j\equiv 0 \bmod (p-1)$, or both $l=0$ and 
$j\equiv -1 \bmod (p-1)$. 
\end{definition}

\begin{lemma}\label{AllSpModules}
$\mathcal{A}^{[S]}$ is the set of all 
finite $\mathrm{S}_p$-submodules of 
$\mathrm{D}(p,\C)$.
\end{lemma}
\begin{proof}
Follows from Lemmas~\ref{tConjugacyonsqModules} 
and \ref{SpModulesp}.
\end{proof}
In particular, for $q\neq p$, the $\Omega_nX_{\{q\}}$ 
are all the finite $\mathrm{S}_p$-submodules of
$X_{\{q\}}$.

\begin{lemma}\label{H1XandD}
$H^1(\mathrm{S}_p, X_{\{2\}})=0$ and 
$H^1(\mathrm{S}_p, D_{\{q\}})=0$ if $q$ is odd.
\end{lemma}
\begin{proof}
By Lemmas~\ref{Shapiro} and \ref{Spminus1},
$H^1(\mathrm{S}_p, D_{\{q\}}) \cong 
 H^1(\mathrm{S}_{p-1}, Z_{\{q\}})$. 
Now $H^1(\mathrm{S}_n, Z_{\{q\}}) = 
\mathrm{Hom}(\mathrm{S}_n/\mathrm{S}_n', Z_{\{q\}})$.
Thus $H^1(\mathrm{S}_p, D_{\{q\}})=0$ if $q$ is odd,
whereas $H^1(\mathrm{S}_p, D_{\{2\}})\cong  
C_2$. Since $D_{\{2\}}= X_{\{2\}}\times Z_{\{2\}}$ 
implies that $H^1(\mathrm{S}_p,D_{\{2\}})= 
H^1(\mathrm{S}_p,X_{\{2\}})\oplus 
 H^1(\mathrm{S}_p,Z_{\{2\}})$, the rest of 
the lemma is clear.
\end{proof}

\begin{definition}
Let $r= (1,2)\in \mathrm{S}_p$; so
$\mathrm{S}_p = \langle s, r\rangle$. 
Define $\mathcal R$ to be the set of groups 
$\langle s, r, A\rangle$ 
and $\langle s, rz_{2^{n+1}}, A\rangle$
where $A\in \mathcal{A}^{[S]}$ and
$|A\cap Z_{\{2\}}|=2^n$, for all $n\geq 0$. 
\end{definition}

\begin{proposition}
If $G$ is a finite subgroup of $\Mo(p,\C)$ 
 with permutation part $\mathrm{S}_p$, 
 then $G$ is $D$-conjugate to a group in 
 $\mathcal R$.
\end{proposition}
\begin{proof} 
Let $G$ be a hub group, i.e., $G\leq \mathrm{S}_pD_\pi$
and $s\in G$
(see Proposition~\ref{FirstNonSolvableReduction}).
We have $D_\pi = B\times Z_{\{2\}}$ where 
$B = X_{\{2\}} \prod_{q\in \pi \setminus \{2\}} D_{\{q\}}$.
By Lemma~\ref{H1XandD}, $H^1(\mathrm{S}_p,B) = 0$, and 
$B/(G\cap B) \cong B$ as $\mathrm{S}_p$-modules by 
Lemmas~\ref{ZetaLem} and \ref{AllSpModules}. 
Proposition~\ref{1stCoh} then gives $d\in D$ such that  
$G^d = H . (B\cap G)$ where 
$H=G^d\cap \mathrm{S}_pZ_{\{2\}}$ and $s\in H$. 
Thus $H = \langle s, rz, Z_{2^n}\rangle$ for some $n$ 
and $z\in Z_{\{2\}}$. So $G^d\in \mathcal R$ 
because $z^2=(rz)^2 \in Z_{2^n}$. 
\end{proof}

\begin{lemma}
A finite subgroup of $\Mo(p,\C)$ with permutation 
part $\mathrm{S}_p$ is reducible if and only if 
its diagonal subgroup is scalar. 
\end{lemma}
\begin{proof}
If $G\in \mathcal R$ and $G\cap D\leq Z$ then
$G\leq \mathrm{S}_pZ$ is reducible.
\end{proof}
\begin{lemma}\label{NonConjSp}
Distinct  irreducible 
groups in $\mathcal R$ are not $\GL(p,\C)$-conjugate.
\end{lemma}
\begin{proof}
Let $G\in \mathcal R$ be irreducible with diagonal 
subgroup $A$, and suppose that $G^h\in \mathcal R$ for 
some $h\in \GL(p,\C)$. 
By Theorem~\ref{GLConjugateImpliesSamePP}, 
$A^h=G^h\cap D$. 
Then $G^h\cap D = A$ by Theorem~\ref{Protoconjugacy}.
The two groups in $\mathcal R$
with diagonal subgroup $A$ are not conjugate
(their images under the determinant map have different 
$2$-parts), so $G^h=G$. 
\end{proof}

We next delete the reducible groups from $\mathcal R$.
\begin{definition}\label{RStarDef}
Let $\mathcal{R}^*$ be the subset 
of $\mathcal R$ 
consisting of all groups $G$ such that $A=G\cap D\in 
\mathcal{A}^{[S]}$ as in Definition~\ref{MathcalAsDefn}
satisfies one of the following:
\begin{itemize}
\item[(i)] $j\neq 0$ and $j\equiv 0 \bmod (p-1)$;
\item[(ii)] $j\equiv -1 \bmod (p-1)$ and $l=0$;
\item[(iii)] $j=l=0$ and ${\sf N}_A\neq 0$.
\end{itemize}
\end{definition}

The main problems for $T=\mathrm{Sym}(p)$ are now solved.
\begin{theorem}
$\mathcal{R}^*$ is a classification of 
the finite irreducible subgroups of $\Mo(p,\C)$ with 
permutation part $\mathrm{Sym}(p)$.
\end{theorem}

\section{Permutation part $\mathrm{Alt}(p)$}
This section incidentally disposes of all 
degrees at most $5$. 
%The discussion 
%mirrors that in Section~\ref{PermPartSymp}.  
Degree $5$ requires added care.

Let $p\geq 5$ and $\mathrm{A}_n:= \mathrm{Alt}(n)$. 
\begin{proposition}
\label{ApModuleIffSpModule}
A finite $\mathrm{A}_p$-submodule of $\Di(p,\C)$ is an
$\mathrm{S}_p$-module.
\end{proposition}
\begin{proof}
By Lemma~\ref{SpModulesp}, we need only show that 
$M:=\Omega_1 X_{\{q\}}$ is an irreducible 
$\mathrm{A}_p$-module for primes $q\neq p$. 
By Proposition~\ref{410Bacskai},  
Lemma~\ref{411Bacskai}, and the proof of 
Theorem~\ref{qsSubmodules}, $M$ is a direct product of 
$v$ irreducible pairwise non-isomorphic 
$\langle s\rangle$-submodules $X_{q,1}^{(i)}$. 
 We infer from Lemma~\ref{tConjugacyonsqModules} that
 $M$ is an irreducible $\langle s,t^2\rangle$-module 
when $v$ is odd. 
Let $v$ be even. As an $\mathrm{A}_p$-module, if $M$ 
were reducible then it would be 
the direct sum of its $\langle s,t^2\rangle$-submodules 
$\prod_{i \ \mathrm{odd}} X_{q,1}^{(i)}$
and $\prod_{i \ \mathrm{even}} X_{q,1}^{(i)}$.
This contradicts Proposition~\ref{XqIndecomposable}.
\end{proof}

\begin{lemma}\label{H1A5C3}
$H^1(\mathrm{A}_5, X_{\{3\}}) = C_3$
and $H^1(\mathrm{A}_p, D_{\{q\}}) = 0$ for 
$(p,q) \neq (5,3)$.
\end{lemma}
\begin{proof}
Cf.~the proof of Lemma~\ref{H1XandD}.
\end{proof}

\begin{definition}
Let $w=(1,2,3)\in \mathrm{P}(p)$, so
$\mathrm{A}_p = \langle s, w\rangle$.
Let $\mathcal{U}_0$ be the set of 
groups $\langle s, w, A\rangle$ for 
$A\in \mathcal{A}^{[S]}$.
\end{definition}

\begin{proposition}\label{CompletenessApGeneral}
If $p>5$ and $G$ is a finite subgroup of $\Mo(p,\C)$
with permutation part $\mathrm{A}_p$, then  $G$ is 
$D$-conjugate to a group in $\mathcal{U}_0$.
\end{proposition}
\begin{proof}
Since $H^1(\mathrm{A}_p,D_{\pi}) = 0$ by 
Lemma~\ref{H1A5C3}, and we may take 
$\zeta = \eta = \pi$ in Lemma~\ref{ZetaLem},  
$G$ is conjugate to 
$\langle \mathrm{A}_p, G\cap D\rangle$
by Proposition~\ref{1stCoh}. 
\end{proof}

\begin{lemma}\label{U0NonConj}
Distinct irreducible groups in $\mathcal{U}_0$ 
are not $\GL(p,\C)$-conjugate.
\end{lemma}
\begin{proof}
Cf.~the proof of Lemma~\ref{NonConjSp}. 
Proposition~\ref{ApModuleIffSpModule} comes 
into play; reducibility of $G$ is again 
 equivalent to $G\cap D\leq Z$, and there is 
 a single group in $\mathcal{U}_0$
with a given diagonal subgroup.
\end{proof}

\begin{definition}
Let $\mathcal{U}_0^*$ be the set of 
$\langle s, w, A\rangle\in \mathcal{U}_0$ 
such that one of (i)--(iii) 
as in Definition~\ref{RStarDef} 
holds for $A\in \mathcal A^{[S]}$.
\end{definition}

Thus, $\mathcal{U}_0^*$ is the subset of 
irreducible groups in $\mathcal{U}_0$. 
\begin{theorem}
If $p\geq 7$ then $\mathcal{U}_0^*$ is a 
classification of the finite irreducible 
 subgroups of $\Mo(p,\C)$ with permutation 
 part $\mathrm{Alt}(p)$.
\end{theorem}
\begin{proof}
Proposition~\ref{CompletenessApGeneral}
 and Lemma~\ref{U0NonConj} show that 
$\mathcal{U}_0^*$ is complete and irredundant.
\end{proof}

\subsection{Degree $5$}
Now fix $p=5$. 
\begin{definition} 
Let $c_n=\mathrm{diag}(1,\epsilon_n, 
\epsilon_n^{-1}, \epsilon_n^{-1},\epsilon_n)$
where $\epsilon_n= e^{2\pi\mathrm{i}/3^n}$.
For $i\in \{ 1,2\}$, define 
$\mathcal{U}_i$ to be the set of groups 
$\langle s, wc_{n+1}^i, A\rangle\leq \Mo(5,\C)$ 
where $A\in \mathcal{A}^{[S]}$ 
and $A\cap X_{\{ 3\}} = \Omega_n X_{\{3\}}$, as
$n$ ranges over the non-negative integers.
\end{definition}

\begin{remark}
$\langle s, wc_{n+1}^i\rangle$ has diagonal
subgroup $\Omega_n X_{\{3\}}$.
\end{remark}

\begin{lemma}\label{WiDistinct}
The $W_i = \langle s, wc_1^i \rangle$ for 
$i \in \{ 0,1,2\}$ are isomorphic to 
 $\mathrm{A}_5$, 
and no two of these groups are conjugate in 
$D\mathrm{A}_5$.
\end{lemma}
\begin{proof}
Obviously $W_0=\mathrm{A}_5$ is reducible.
Also, $W_1$ and $W_2$ correspond to the  ordinary
irreducible character of $\mathrm{A}_5$ of 
 degree $5$ (hence they are $\GL(5,\C)$-conjugate).
If $W_1$ and $W_2$ were $D\mathrm{A}_5$-conjugate,
 then they would be $D$-conjugate; but 
 $s^d\not \in W_2$ for non-scalar $d\in D$.
\end{proof}

\begin{theorem}\label{A5Completeness}
A finite subgroup $G$ of $D \mathrm{A}_5$
such that $\phi(G) = \mathrm{A}_5$
is conjugate to a group in 
$\mathcal{U}_0\cup\mathcal{U}_1 \cup \mathcal{U}_2$.
\end{theorem}
\begin{proof}
By Proposition~\ref{1stCoh}, and
Lemmas~\ref{ZetaLem} and \ref{H1A5C3}, 
we may suppose that the hub group $G$ is 
in $M:=X_{\{3\}}\mathrm{A}_5$.
So $G\cap D = \Omega_nX_{\{3\}}$ for some $n$.  
Let $\tau$ be the surjective endomorphism of $M$
that is the identity on $\mathrm{A}_5$ and maps 
$x\in X_{\{3\}}$ to $x^{3^n}$.  Note that
$\ker \tau = G\cap D$.

The $W_i$ and $\tau(G)$
are all complements of $X_{\{3\}}$ in $M$. 
Since $|H^1(\mathrm{A}_5,X_{\{3\}})| = 3$, 
  Lemma~\ref{WiDistinct}
implies that $\tau(G)$ is $M$-conjugate, i.e.,
$\tau(M)$-conjugate, to 
$W_i=\tau(\langle s,wc_{n+1}^i\rangle)$ for some $i$. 
Therefore $G$ is $M$-conjugate to 
$\langle s,wc_{n+1}^i\rangle$. 
\end{proof}

\begin{lemma}
Each group in $\mathcal{U}_2$ is 
conjugate to a group in $\mathcal{U}_1$.
\end{lemma}
\begin{proof}
Here $t=(1,2,4,3)$.
 Matrix multiplication establishes that 
$(wc_n)^{s^3wc_ns^2} \equiv w^t c_n^{2t}$ 
 modulo the diagonal subgroup 
 $\Omega_{n-1}X_{\{3\}}$ of 
 $\langle s, w c_n\rangle$.
\end{proof}

\begin{lemma}
All groups in $\mathcal{U}_1$ are irreducible.
\end{lemma}
\begin{proof}
Let $G=\langle s,wc_{n+1},A\rangle \in \mathcal{U}_1$. 
If $n= 0$ then $G$ contains the irreducible group 
$W_1$.
If $n\geq 1$ then $A$ is non-scalar. 
\end{proof}
\begin{lemma}
Distinct groups in $\mathcal{U}_0\cup \mathcal{U}_1$
are not $\GL(5,\C)$-conjugate.
\end{lemma}
\begin{proof}
We proceed as in the proofs of Lemmas~\ref{NonConjSp} 
and \ref{U0NonConj}. Suppose that 
$\langle s, wc_{n+1}, \allowbreak A\rangle^d =
 \langle s, w, A \rangle$  
for some $d\in D$, where 
$A\cap X_{\{3\}}= \Omega_n X_{\{3\}}$. 
Then $[w,d]c_{n+1}\in A\cap X$.
However, $[w,d]c_{n+1}$ has a diagonal 
entry of order $3^{n+1}$.
\end{proof}
\begin{theorem}
$\mathcal{U}_0^*\cup \mathcal{U}_1$ is a classification 
of the finite irreducible subgroups of $\Mo(5,\C)$ 
with permutation part $\mathrm{Alt}(5)$.
\end{theorem}

\section{Degrees greater than $5$}
\label{DegreesGreaterThan5}

Let $G$ be a  finite irreducible subgroup 
of $\Mo(p,\C)$ with permutation part $T$.
In previous sections we classified
all $G$ such that $T$ is compulsory.
A member of the non-compulsory 
`projective' family, $\SL(3,2)$, is 
self-normalizing in $\Sym(7)$, hence is the only 
non-compulsory $T$ for $p=7$.
In degrees $11$ and $23$, the non-compulsory 
$T$ are $M_{11}$, $\PSL(2,11)$, and $M_{23}$. 

\subsection{Degree $7$}

Let $p=7$ and $T\cong \SL(3,2)$. Thus $\pi=\{2,3\}$.
\begin{definition}
 $V\cong \SL(3,2)$ is the subgroup of $\mathrm{S}_7$
generated by $s$ and $v=(1,2)(3,5)$.
\end{definition}

\begin{lemma}\label{OddqVmodules}
If $q$ is an odd prime then a finite $V$-submodule 
of $D_{\{q\}}$ is an $\mathrm{S}_7$-module.
\end{lemma}
\begin{proof}
We combine Lemma~\ref{SpModulesp},  
Corollary~\ref{qNotInpiModules}, and 
Proposition~\ref{StructureofXqr} 
(for $q=3$, as $X_{\{3\}} = X_{\{3\}}^{(1)}$).
\end{proof}

The $V$-submodule structure of 
$X_{\{2\}}= X_{\{2\}}^{(1)} X_{\{2\}}^{(2)}$
 is less tractable. 
Here we resume the conventions 
of Section~\ref{ModulesOfpCoprimeOrder},
fixing $q=2$.
In $\Z_2[\mathrm{x}]$, 
$f(\mathrm{x}) = 
\mathrm{x}^6+ \mathrm{x}^5+\mathrm{x}^4+
\mathrm{x}^3+ \allowbreak  \mathrm{x}^2+\mathrm{x}+1$ 
factorizes as the product of irreducibles
$f_1$, $f_2$, with integer polynomial 
approximations 
\[
\begin{array}{ll}
f_{1,1}(\mathrm{x})= \mathrm{x}^3
+ \mathrm{x}+1 & \qquad f_{2,1}(\mathrm{x})= 
\mathrm{x}^3 + \mathrm{x}^2+1\\
f_{1,2}(\mathrm{x})= \mathrm{x}^3
+ 2\mathrm{x}^2+\mathrm{x}+3 
& \qquad f_{2,2}(\mathrm{x})= \mathrm{x}^3
+ 3\mathrm{x}^2+2\mathrm{x}+3
\end{array}
\]
(see the proof of Lemma~\ref{Hensellike} for 
the method to calculate each $f_{i,n}$).
\begin{lemma}\label{VSubmodules}\
\begin{itemize}
\item[{\rm (i)}]
$X_{2,m}^{(1)}X_{2,n}^{(2)}$ 
is a $V$-module if and only if $m=n$ or $m=n+1$.
\item[{\rm (ii)}]
$X_{2,m}^{(1)}X_{2,n}^{(2)}$ 
is an $\mathrm{S}_7$-module if and only if $m=n$.
\end{itemize}
\end{lemma}
\begin{proof}
Since $X_{2,1}^{(1)}$
is annihilated by $f_{1,1}(s)$,
the $\F_2$-space $X_{2,1}^{(1)}$ has
basis $\{ x_1,x_1^s,x_1^{s^2}\}$ where
$x_1:=x_{2,1}^{(1)}= \mathrm{diag}(-1,-1,-1,1,-1,1,1)$.
This basis maps to another under action by $v$. Hence 
$X_{2,1}^{(1)}$ is a $V$-module.

Clearly $X_{2,n}^{(1)}
X_{2,n}^{(2)}=\Omega_nX_{\{2\}}$ is a $V$-module; 
as is $X_{2,n+1}^{(1)} X_{2,n}^{(2)}$, being the 
inverse image of $X_{2,1}^{(1)}$ under the 
endomorphism $\kappa^n$ on
$X_{\{2\}}$ that maps $x$ to $x^{2^n}$.

Suppose that $m<n$ and $X_{2,m}^{(1)}X_{2,n}^{(2)}$ 
is a $V$-module. Then $X_{2,1}^{(2)}= 
\kappa^{n-1}\big(X_{2,m}^{(1)} X_{2,n}^{(2)}\big)$ 
is a $V$-module. But $\Omega_1X_{\{2\}} = X_{2,1}^{(1)}
\times X_{2,1}^{(2)}$ is $V$-indecomposable by 
Proposition~\ref{XqIndecomposable}.

Let $x_2= x_{2,2}^{(1)}= \mathrm{diag}(-\mathrm{i},
-\mathrm{i}, \mathrm{i}, -1, -\mathrm{i}, 1,1 )$.
Observe that $x_2^v\not \in X_{2,2}^{(1)}$, as 
$x_2^v$ is not annihilated by $f_{1,2}(s)$.
So $X_{2,2}^{(1)}$ is not a $V$-module.
However, if $m>n+1$ and  
$X_{2,m}^{(1)}X_{2,n}^{(2)}$ were a $V$-module, then 
$X_{2,2}^{(1)}= 
\kappa^{m-2}\big(X_{2,m}^{(1)}X_{2,n}^{(2)}\big)$ 
would be one too. This rules out the final possibility
for $(m,n)$.
\end{proof}

\begin{definition}
Let $\mathcal{A}^{[V]}$ be the set of 
$A\in \mathcal A$ in degree $7$ 
for which the following hold.
\begin{itemize}
\item[(i)] If $A\cap D_{\{7\}}=Y_{j,k,l}$,
 then either $j\equiv 0 \bmod p-1$, 
or both $l=0$ and $j\equiv -1 \allowbreak \bmod p-1$.
\item[(ii)] Either ${\sf N}_A = {\sf N}_A^t$, 
or $A\cap X_{\{2\}} = X_{2,n+1}^{(2)}X_{2,n}^{(2)}$
for some $n$ and ${\sf N}_A$ agrees with
 ${\sf N}_A^t$ in each row apart from the row 
for $q=2$.
\end{itemize}
\end{definition}

\begin{proposition}
$\mathcal{A}^{[V]}$ is the set of all finite 
$V$-submodules of $D$.
\end{proposition}

Since $V\cap \mathrm{S}_6 \cong \mathrm{S}_4$,
Lemmas~\ref{Shapiro} and \ref{Spminus1} give
the following.
\begin{lemma}\label{VCohom}
$H^1(V,D_{\{q\}})= 0$ if $q$ is odd.
\end{lemma}

Next we carry out some matrix arithmetic.
\begin{lemma}\label{AlphaBetaEtc}
If $d\in \mathrm{D}(7,\C)$ and 
 $\langle s, vd\rangle \cap D
\leq X_{2,1}^{(1)}$, then $d^2=1$.
\end{lemma}
\begin{proof}
Let $d= \mathrm{diag}(a_1,\ldots , a_7)$. We 
evaluate the containment of 
 $(vd)^2, (svd)^4, \allowbreak (s^2vd)^3$ 
 in $\Omega_1 D_{\{2\}}$ 
and $(vd)^2, (svd)^4$ in $\ker f_{1,1}(s)$ to 
get a system of equations in the $a_i$ whose 
simultaneous solution implies that 
each $a_i$ is $\pm 1$. 
\end{proof}

\begin{definition}
Let $g_n=\mathrm{diag}(1,1,1,\epsilon_n, 1,
\epsilon_n^{-1},1)$ and 
$h_n=\mathrm{diag}(\epsilon_n,\epsilon_n,
1, \epsilon_n^{-1}, 1, 1 , \epsilon_n^{-1})$ 
where 
$\epsilon_n = e^{\pi {\rm i}/2^{n-1}}$.
\end{definition}

\begin{remark}
The diagonal subgroup of $\langle s, v g_1\rangle$ 
 is $X_{2}^{(1)}$, and 
 $\langle s, v h_1\rangle\cong V$ is irreducible.
\end{remark}

\begin{lemma}\label{SimAlphaBetaEtc}
Let $d\in \mathrm{D}(7,\C)$. 
If $\langle s, vd\rangle \cap D
= X_{2,1}^{(1)}$ 
(resp., $\langle s, vd\rangle \cap D =1$), 
then $d\in X_{2,1}^{(1)}$ or $g_1X_{2,1}^{(1)}$
(resp., $d\in \langle h_1\rangle$).
\end{lemma}
\begin{proof}
Follows from Lemma~\ref{AlphaBetaEtc} and
calculations similar to those in its proof. 
\end{proof}
\begin{definition}\label{CalVDefn}
Let
\[
\begin{array}{l}
\mathcal{V}_0 = \{ \langle s, v, A\rangle \mid 
 A \in \mathcal{A}^{[V]}\} , \\
\mathcal{V}_1 = {\textstyle \bigcup_{n\geq 1}} 
\big\{ \langle s, v g_{n}, 
A\rangle \mid  A\in 
\mathcal{A}^{[V]}\setminus \mathcal{A}^{[S]}
\text{ and } 
A\cap X_{\{2\}}^{(1)} = X_{2,n}^{(1)} \big \} , \\
\mathcal{V}_2 = {\textstyle \bigcup_{n\geq 0}} 
\big\{ \langle s, 
v h_{n+1}, 
A\rangle \mid 
A \in \mathcal{A}^{[S]} \text{ and } 
A\cap X_{\{2\}}^{(1)} = 
X_{2,n}^{(1)}  \big\}  .
\end{array}
\]
Then let $\mathcal{V} = 
\mathcal{V}_0\cup \mathcal{V}_1 \cup \mathcal{V}_2$.
\end{definition}

\begin{theorem}
A finite subgroup of $\Mo(7,\C)$ with permutation part
$\SL(3,2)$ is $D\mathrm{S}_7$-conjugate to a group in 
$\mathcal{V}$.
\end{theorem}
\begin{proof}
(Cf.~the proof of Theorem~\ref{A5Completeness}.)
By Lemmas~\ref{OddqVmodules} and \ref{VCohom},
we consider a finite hub group $G\leq VX_{\{2\}}$ 
with  $\phi(G)=V$ and $s\in G$.  
Lemma~\ref{VSubmodules} indicates that 
$\Omega_nX_{\{2\}} \leq D\cap G 
\leq \Omega_{n+1}X_{\{2\}}$ for some $n$.

Let $\kappa$ be the group endomorphism of 
$\mathrm{S}_7X_{\{2\}}$ that squares elements of 
$X_{\{2\}}$ and is the identity on $\mathrm{S}_7$. 
If $G\cap D = X_{2,n+1}^{(1)}X_{2,n}^{(2)}$
then $\kappa^n(G)\cap D = X_{2,1}^{(1)}$. 
By Lemma~\ref{SimAlphaBetaEtc},
$\kappa^n(G) = VX_{2,1}^{(1)}$ or 
$\langle s, vg_1\rangle$. 
The preimages of these groups under $\kappa^n$
are in $\mathcal{V}_0\cup \mathcal{V}_1$. 

Suppose that $G\cap D =\Omega_nX_{\{2\}}$. 
By Lemma~\ref{SimAlphaBetaEtc}, 
$\kappa^n(G)$ is then $V$ or $\langle s, vh_1\rangle$.
Hence $G\in \mathcal{V}_0\cup \mathcal{V}_2$.
\end{proof}

\begin{lemma}
$G\in \mathcal V$ is irreducible if and only 
if $G\not \in \mathcal V_0$ or $G\cap D \not \leq Z$.
\end{lemma}
\begin{proof}
Suppose that $G\in \mathcal{V}_1\cup \mathcal{V}_2$ 
and $G\cap D\leq Z$. Then $G\in \mathcal{V}_2$ by
Definition~\ref{CalVDefn}, and we know that 
$\langle s, vh_1\rangle\leq G$ 
is irreducible.
\end{proof}

\begin{definition}
Let $\mathcal{V}^*$ be the sublist of $\mathcal V$ 
that excludes all members of $\mathcal{V}_0$ with 
scalar diagonal subgroup.
\end{definition}

\begin{theorem}
$\mathcal{V}^*$ is 
a classification of the finite irreducible 
subgroups of $\Mo(7,\C)$ with permutation part $\SL(3,2)$. 
\end{theorem}
\begin{proof}
To prove irredundancy of $\mathcal{V}^*$, suppose that
$G\in \mathcal{V}_0$ and $H\in \mathcal{V}_1\cup
\mathcal{V}_2$ are $D\mathrm{S}_7$-conjugate, hence
$DV$-conjugate (as $N_{\mathrm{S}_7}(V) = V$),
with the same diagonal subgroup $A$. 
Let $A\cap X_{\{2\}}^{(1)} = X_{2,n}^{(1)}$.
Then there is $d\in \mathrm{D}(7,\C)$ 
such that $[v,d]b\in A$ where 
$b=g_{n}$ ($n\geq 1$) or $h_{n+1}$ ($n\geq 0$).  
Since the last two (and fourth) diagonal entries of 
$[v,d]$ are $1$s, $b$ cannot be $h_{n+1}$.
Also, $\kappa^{n-1}([v,d]g_n)\not \in X_{2,1}^{(1)}$ 
for any $d$.
\end{proof}

\subsection{Degree $11$}
\label{Degree11Subsection}

There are two non-isomorphic permutation 
parts of degree $11$, each of which is is 
self-normalizing in $\mathrm{Sym}(11)$.

\begin{definition}
Let $w_1$, $w_2$ be the permutation matrices 
corresponding respectively 
to $(1,7)(2,3)(4,8)(5,9)$, 
$(1,3)(2,8)(4,7)(5,6)\in \mathrm{Sym}(11)$.
Then let $P = \langle s, w_1\rangle$ and
$Q = \allowbreak \langle s, w_2\rangle$. 
\end{definition}
In fact $P\cong \mathrm{PSL}(2,11)$, $Q\cong 
M_{11}$, and $P\leq Q$. For both groups, 
$\pi =\{ 2,3,5\}$.
\begin{lemma}\label{PModules1}
If $q\neq 3$ then the finite $P$-submodules 
of $D_{\{q\}}$ are $\mathrm{S}_{11}$-modules.
\end{lemma}
\begin{proof}
For $q\not \in \{3,5\}$, cf.~the proof of 
Lemma~\ref{OddqVmodules}. Inspecting
the actions of $f_{1,1}(s)$ and $f_{2,1}(s)$ 
on $(x_{5,1}^{(1)})^{w_1}$ and 
$(x_{5,1}^{(2)})^{w_1}$,
we see that $X_{5,1}^{(1)}$ and
$X_{5,1}^{(2)}$ are not $P$-modules.
Thus $\Omega_1 X_{\{5\}}$
is irreducible.
\end{proof}

\begin{lemma}\label{PModules2}
$X_{3,m}^{(1)}X_{3,n}^{(2)}$ is a $P$-module 
if and only if $n=m$ or $n=m+1$.
\end{lemma}
\begin{proof}
Cf.~the proof of Lemma~\ref{VSubmodules}; 
$X_{3,1}^{(2)}$ is a $P$-module, 
while $X_{3,2}^{(2)}$ is not. 
\end{proof}
Lemmas~\ref{PModules1} and \ref{PModules2} 
inform the next definition.
\begin{definition}
Let $\mathcal{A}^{[P]}$ be the set of 
$A\in \mathcal A$ in degree $11$ for which 
the following hold.
\begin{enumerate}
\item If $A\cap D_{\{p\}} = Y_{j,k,l}$, then 
either $j \equiv 0 \bmod p-1$, or both
$l=0$ and $j+1\equiv 0 \allowbreak \bmod p-1$.
\item Either ${\sf N}_A={\sf N}_A^t$, or 
$A\cap X_{\{3\}} = X_{3,n}^{(1)}X_{3,n+1}^{(2)}$ 
for some $n\geq 0$ 
and ${\sf N}_A$ agrees with ${\sf N}_A^t$
in each row apart from the row for $q=3$.
\end{enumerate}
\end{definition} 

Moving between $\mathcal{A}^{[V]}$ and 
$\mathcal{A}^{[P]}$, the roles of 
$X_{q,*}^{(1)}$ and
 $X_{q,*}^{(2)}$ are switched as 
the critical prime $q$ switches between
 $2$ and $3$. The Hasse diagram of the 
$T$-submodule lattice of $X_{\{q\}}$ is a
zig-zag chain.
\begin{proposition}\
\begin{itemize}
\item[{\rm (i)}]
$\mathcal{A}^{[P]}$ is the set of all finite 
$P$-submodules of $D$.
\item[{\rm (ii)}]
$\mathcal{A}^{[S]}$ is the set of all finite 
$Q$-submodules of $D$.
\end{itemize}
\end{proposition}
\begin{proof}
Part (ii) follows from part (i): $P\leq Q$, and 
$X_{3,n}^{(1)}X_{3,n+1}^{(2)}$ is not a $Q$-module
because $(x_{3,1}^{(2)})^{w_2} \not \in X_{3,1}^{(2)}$.
\end{proof}

Since $P\cap \mathrm{S}_{10}\cong \mathrm{A}_5$ 
and $Q\cap \mathrm{S}_{10}\cong \mathrm{A}_6\cdot 2$, 
we deduce the following.
\begin{lemma}\label{1stCPandQ}\
\begin{itemize}
\item[{\rm (i)}] 
$H^1(P,D_{\{q\}})=0$ for all primes $q$.
\item[{\rm (ii)}] $H^1(Q,D_{\{q\}})=0$ for odd primes
$q$, and $H^1(Q,X_{\{2\}})=C_2$. 
\end{itemize}
\end{lemma}

\begin{lemma}
If $G$ is a finite subgroup of $\Mo(11,\C)$ 
with permutation part $P$, then  
$G$ is $D$-conjugate to $P.(G\cap D)$.
\end{lemma}
\begin{proof}
By Proposition~\ref{1stCoh}, Lemma~\ref{ZetaLem}
(with $T=\mathrm{S}_{11}$), and 
Lemma~\ref{1stCPandQ}~(i), we may suppose that 
$G\cap D\in \mathcal{A}^{[P]}\setminus
\mathcal{A}^{[S]}$. Let $W$ be the largest 
$\mathrm{S}_{11}$-module in $G\cap D_{\pi}$, 
i.e., $(G\cap D_\pi)/W \cong A:= X_{3,1}^{(2)}$. 
Then $D_\pi/W \cong D_\pi$ 
by Lemma~\ref{ZetaLem}, so that 
 $D_\pi/(G\cap D_\pi)\cong D_\pi/A$.
The short exact sequence
\[
1\rightarrow A \rightarrow D_\pi 
\rightarrow D_\pi/A\rightarrow 1
\] 
gives rise to the fragment
\[
H^1(P,D_\pi) \rightarrow 
 H^1(P,D_\pi/(G\cap D_\pi)) 
\rightarrow H^2(P,A)
\]
of a long exact sequence 
(see \cite[p.~573]{Rotman}).
Since $H^1(P,D_\pi)= H^2(P,A)=0$
by Lemma~\ref{1stCPandQ}~(i) and
\cite[p.~229]{HoltandPlesken}, 
the proof is complete by 
 Proposition~\ref{1stCoh}.
\end{proof}

\begin{theorem}
The set of all $PA$ where $A\in \mathcal{A}^{[P]}$ 
is non-scalar is a classification 
of the finite irreducible subgroups of $\Mo(11,\C)$ 
with permutation part $\PSL(2,11)$. 
\end{theorem}
\begin{proof}
Let $A, B\in \mathcal{A}^{[P]}$.
If $A\leq Z$ then the split extension 
$A\rtimes P$ is reducible. 
If $PA$ and $PB$ are 
$\GL(11,\C)$-conjugate, then $A=B$ by 
Theorems~\ref{GLConjugateImpliesSamePP}
and \ref{Protoconjugacy}, because 
$N_{\mathrm{S}_{11}}(P)=P$.
\end{proof}

To conclude degree $11$, we list the groups 
with permutation part $M_{11}$.
\begin{definition}
Let $d_n=\mathrm{diag}(\epsilon_n, \epsilon_n^{-1}, 
\epsilon_n^{-1}, 1, \epsilon_n^{-1},
\epsilon_n ,1, \epsilon_n, 1,1,1)$
where $\epsilon_n= 
\allowbreak e^{\pi {\rm i}/2^{n-1}}$. 
Define $\mathcal Q$ to be the set of groups 
$\langle s, w_2, A\rangle$ and
 $\langle s, w_2d_{n+1}, A\rangle$
for $A\in \mathcal{A}^{[S]}$ such that 
$A\cap X_{\{2\} } = \Omega_n X_{\{2\}}$,
for all $n\geq 0$. 
\end{definition} 

\begin{theorem}
The subset of  $\mathcal Q$ 
that excludes (only) the groups 
$\langle s, w_2, A\rangle$, 
where $A\leq Z$, is a classification 
of the finite irreducible 
subgroups of $\Mo(11,\C)$ with 
permutation part $M_{11}$. 
\end{theorem}
\begin{proof}
Let $G$ be a hub group in $QX_{\{2\}}$
(by way of Lemma~\ref{1stCPandQ}~(ii),
Proposition~\ref{1stCoh}, and Lemma~\ref{ZetaLem}). 
Repeated squaring on $G$
reduces to $Q$ or $\langle s, w_2 d_1\rangle$, 
the only copies of $M_{11}$ in 
$QX_{\{2\}}$ up to conjugacy 
(see Lemma~\ref{1stCPandQ}~(ii)). 
These have preimages in $\mathcal Q$.

Since $\langle s, w_2d_1\rangle$ is irreducible, 
the reducible groups in  $\mathcal Q$ are the $QA$
with $A$ scalar.

Suppose that $A\cap X_{\{ 2\}} = \Omega_n X_{\{ 2\}}$
 and $G = QA$ is $DQ$-conjugate, hence 
$D$-conjugate, to $\langle s, w_2d_{n+1}, A\rangle$.
Then there is a diagonal matrix $b$ such that
$b^{1-sw_2} = d_{n+1}\in A\cap X$. However,   
$b^{1-sw_2}$ has diagonal entries of order $2^{n+1}$.
\end{proof}

\subsection{Degree 23}

Let $p=23$. A non-compulsory transitive subgroup 
of $\mathrm{P}(23)$
is conjugate to the group $Q\cong M_{23}$ 
generated by $s$ and
\[
(1,3)(4,19)(5,17)(6,9)(7,8)(10,16)(12,15)(13,18).
\]
(We recycle the notation 
$Q$ from Section~\ref{Degree11Subsection}.)
\begin{definition}
Let $\mathcal{A}^{[Q]}$ be the sublist of 
$A\in \mathcal A$ in degree $23$ for which the 
following hold.
\begin{enumerate}
\item If $A\cap D_{\{p\}} = Y_{j,k,l}$, then either 
$j \equiv 0 \mod p-1$, or both $l=0$ and $j+1\equiv 0 \mod
p-1$.
\item Either ${\sf N}_A={\sf N}_A^t$, or 
$A\cap X_{\{2\}} = X_{2,n+1}^{(1)}X_{2,n}^{(2)}$ for
some $n\geq 0$ and ${\sf N}_A^t={\sf N}_A$ apart
from the row for $q=2$.
\end{enumerate}
\end{definition} 
 Proofs of the next three results are left as exercises.
\begin{lemma}
$\mathcal{A}^{[Q]}$ is the set of all finite
$Q$-submodules of $\mathrm{D}(23,\C)$.
\end{lemma}
\begin{lemma}
If $q$ is prime then $H^1(Q, D_{\{q\}}) = 0$.
\end{lemma}
\begin{theorem}
The set of all $AQ$ for non-scalar 
$A\in \mathcal{A}^{[Q]}$ 
is a classification of the finite irreducible 
subgroups of $\mathrm{M}(23,\C)$ with 
permutation part $M_{23}$.
\end{theorem}

\section{Overview}\label{Overview}
We have completely and irredundantly classified up to 
conjugacy in $\GL(p,\C)$ all finite irreducible 
monomial subgroups that
\begin{itemize}
\item are solvable;
\item have permutation part 
containing $\mathrm{Alt}(p)$;
\item are non-solvable in degrees
$p\in \{7,  11, 23\}$.
\end{itemize}
\noindent
Hence, we have classified the groups for 
$p\leq 11$, $p=23$, and the infinitely 
many $p$ not of the form $(q^d-1)/(q-1)$ 
where $q$ is a prime power. 

Our methodology may be used to settle all prime 
degrees $p<31$. If $p=19$ or $29$ then the 
permutation part $T$ is compulsory.
If $p=13$ or $17$ then the non-compulsory $T$ are 
projective, with one or three possible isomorphism 
types, respectively.

A hermetic classification of the finite 
irreducible subgroups of $\Mo(p,\C)$ for arbitary 
prime $p$ is obstructed by a lack of solutions to 
the $T$-module listing problem (in $D_{ \{ p\}'}$) 
and the extension problem for projective $T$. 
We pose some conjectures,
suggested by existing evidence, 
whose resolution might aid in closing these gaps.
Note that \cite{DZII} is also stymied by
the projective family case, there being
a question about `basic subgroups' of `height' 
greater than $1$~\cite[p.~366]{DZII}.

Let $T$ be a non-solvable transitive subgroup of 
$\mathrm{Sym}(p)$.
\begin{conjecture}\label{TModuleIffNTModule}
\emph{Every finite $T$-submodule of $D_{\{p\}'}$ is an 
$N_{\mathrm{Sym}(p)}(T)$-submodule.}
\end{conjecture}

Suppose that $\SL(d,q)\unlhd T\leq 
\Sigma \mathrm{L}(d,q)$. 
If $q$ is prime then $N_{\mathrm{Sym}(p)}(T)=T$.
The only $p$ such that $5<p<10^6$ and 
$p=(q^d-1)/(q-1)$ with $q$ composite are 
$p=17$, $73$, $257$, $757$, $65537$,
and $262657$. 
Since the finite $\SL(2,16)$-submodules of 
$\mathrm{D}(17,\C)$ are $\mathrm{Sym}(17)$-modules 
(cf.~the proof of Lemma~\ref{OddqVmodules}),
the smallest degree at which
Conjecture~\ref{TModuleIffNTModule}
could fail is $73$. 
This conjecture has a bearing on the 
conjugacy problem (see 
Theorems~\ref{GLConjugateImpliesSamePP}
and \ref{Protoconjugacy}). 

\begin{conjecture}\label{NumberExtensions}
\emph{Every finite $T$-submodule of $D$ has
the same number of $T$-extensions in 
$\widetilde{\mathrm{M}}(p,\C)$
up to $\widetilde{\mathrm{M}}(p,\C)$-conjugacy.}
\end{conjecture}
The number $n_{\scst T}$ in 
 Conjecture~\ref{NumberExtensions}
is $2$, $1$, $2$, $2$, $1$, $2$, $1$ for 
$T = \mathrm{S}_p$, $\mathrm{A}_p$
($p>5$), $\mathrm{A}_5$, 
$\SL(3,2)$, $\PSL(2,11)$, $M_{11}$, $M_{23}$,
respectively.
In degrees $p\in \{13, 17\}$, the
conjecture is true for $T=\SL(3,3)$ and 
$\SL(2,16)$, with $n_{\scst T}=2$ and $4$, 
respectively; we suppress the proofs. 
Surjective endomorphisms of $TD_\pi$ 
that act identically on $T$ (e.g., denoted 
as powers of $\kappa$ when $p=7$ and $T=V$)  
were used to validate Conjecture~\ref{NumberExtensions}
in the cases so far examined. 
Such maps need not always exist:
it can be shown that there are none for 
$T\cong \SL(5,2)$. 
However, another pattern emerges from the body of
results about $X_{\{q\}}$ for $p\leq 31$
and $q\neq p$.

\begin{conjecture}
\emph{Every indecomposable $T$-submodule of 
$D$ is uniserial.}
\end{conjecture}

Rather than pursuing such conjectures
to ever higher degrees, it seems more fruitful 
to classify primitive groups of moderate prime 
degree.
The non-solvable finite primitive subgroups of 
$\SL(p,\C)$ are listed up to isomorphism 
in \cite{DZI}. Our ultimate goal is a 
(complete, irredundant, explicit) classification 
 of all finite irreducible subgroups of 
$\GL(p,\C)$ for $p\leq 11$ (at least). 
The next section begins this work.

\section{Finite complex linear groups 
of degrees $2$ and $3$}\label{Deg2and3}

Some material in this section pertaining to 
finite primitive subgroups of $\SL(2,\C)$ 
and $\SL(3,\C)$ is common knowledge, tracing back to 
old classifications referenced in 
Section~\ref{Introduction}. A convenient source 
is \cite[Chapters~X, XII]{MillerBlichfeldtDickson}.

The following easy lemma and its corollary assist in
irredundancy proofs.
\begin{lemma} 
\label{CheckByCharacters}
Let $G$ be a finite irreducible subgroup of
$\GL(n,\F)$ where $\F$ is any field. 
Then $\mathrm{Aut}(G)$ has a natural action on the set 
of equivalence classes $[\rho]$ of
faithful irreducible representations 
$\rho\colon G \rightarrow \GL(n,\F)$, defined by 
$[\rho]\theta = [\rho \circ \theta]$. Under this
action,
\begin{itemize}
\item[{\rm (i)}] 
$\mathrm{Stab}_{\Aut(G)}([\rho]) \cong 
N_{\GL(n,\F)}(\rho(G))/C_{\GL(n,\F)}(\rho(G));$
\item[{\rm (ii)}] the 
 orbits are in one-to-one correspondence with the
 $\GL(n,\F)$-conjugacy classes of all irreducible 
subgroups of $\GL(n,\F)$ isomorphic to $G$.
\end{itemize}
\end{lemma}
\begin{corollary}\label{SingleCClNonsolvable}
Let $G$ be a finite absolutely irreducible subgroup 
of $\GL(n,\F)$ that is self-normalizing in $\GL(n,\F)$
 modulo scalars. If there are precisely
%$|\mathrm{Aut}(G)|\big/ |G:Z(G)|$ 
$|\mathrm{Out}(G)|$
inequivalent faithful 
absolutely irreducible representations of $G$ in $\GL(n,\F)$, 
then every absolutely irreducible subgroup
of $\GL(n,\F)$ isomorphic to $G$ is conjugate to $G$. 
\end{corollary}

\subsection{Degree 2}
Some of our generators for the primitive 
groups in degree $2$ are taken from 
\cite[\S \S \! 102--103]{MillerBlichfeldtDickson}.
The others are in $D_{\{p\}}$; 
see Section~\ref{pPowerSubmodules}. 
\begin{definition}
Let 
\[
a= {\small
\textstyle{\frac{1}{2}}
\left[
\renewcommand{\arraycolsep}{.15cm}
 \begin{array}{lr}
\mathrm{i}-1 & \mathrm{i}-1 \\
\vspace{-.55cm} & \\
\mathrm{i}+1& -\mathrm{i}-1 \end{array} \right]},
\qquad 
b= {\small
\textstyle{\frac{1}{\sqrt{2}}}
\left[
\renewcommand{\arraycolsep}{.1cm}
\begin{array}{cc}
1+\mathrm{i} & 0 \\
0 & 1-\mathrm{i} \end{array} \right]}, \qquad
c= {\small
\textstyle{\frac{1}{2}}
\left[
\renewcommand{\arraycolsep}{.1cm}
\begin{array}{cc}
\mathrm{i} & \lambda_{\scst 1} - 
\lambda_{\scst 2}\mathrm{i} \\
-\lambda_{\scst 1} -
\lambda_{\scst 2}\mathrm{i} & 
- \mathrm{i} \end{array}
\right]},
\]
where $\lambda_1 = \frac{1-\sqrt{5}}{2}$ and
 $\lambda_2 = \frac{1+\sqrt{5}}{2}$.
\end{definition}

\begin{theorem}\label{Degree2PrimitivesAll}\mbox{}
\begin{itemize}
\item[{\rm (i)}] 
A finite subgroup $G$ of $\GL(2,\C)$
 is (irreducible) primitive if and only if 
$G\cap \mathrm{D}(2,\C) = Z(G)$ has even order and 
$G/Z(G)$ is isomorphic to 
$\mathrm{Alt}(4)$ or $\mathrm{Sym}(4)$ or
$\mathrm{Alt}(5)$. 
\item[{\rm (ii)}] Let $n\geq 2$ be an even integer. 
\begin{itemize}
\item
For $K= \mathrm{Alt}(4)$ and $K = \mathrm{Sym}(4)$, 
there are precisely two conjugacy classes of groups 
$G\leq \GL(2,\C)$ such that $|Z(G)| = n$ and 
$G/Z(G)\cong K$.
These have representatives
$\langle   a ,  x_4,  z_n   \rangle$, 
$\langle   az_{3n} ,   x_4   \rangle$
when $K=\mathrm{Alt}(4)$, and 
$\langle   a ,  b ,  z_n  \rangle$, 
$\langle   a ,   bz_{2n} ,  z_n \rangle$
when $K=\mathrm{Sym}(4)$.
\item 
Every subgroup of $\GL(2,\C)$ with center of order $n$ 
and central quotient $\mathrm{Alt}(5)$ is conjugate to 
$\langle   a ,  c ,   x_4 ,  z_n \rangle$.
\end{itemize}
\end{itemize}
\end{theorem}
\begin{proof}
 The proofs of parallel statements 
in \cite[Theorem~5.4]{Flannery4FiniteFields} and
\cite[Theorems~5.8, 5.11]{FlanneryOBrien} for 
non-modular absolutely irreducible primitive groups 
transfer with minor adjustments.
There is a finite primitive group 
$H\leq \SL(2,\C)$ such that $GZ=HZ$ and 
so $G/Z(G)\cong H/Z(H)$. The possible isomomorphism
types of $H/Z(H)$ are identified in 
\cite[\S\S \! 102--103]{MillerBlichfeldtDickson}.
We solve central extension problems for subgroups
of $\GL(2,\C)$ using standard $2$-cohomology
to prove completeness.
\end{proof}
\begin{theorem}
The union of $\mathcal{L}_0$ as in 
Theorem{\em ~\ref{p2Done}} and the set of
all groups listed 
in Theorem{\em ~\ref{Degree2PrimitivesAll}~(ii)} 
as $n$ ranges over the positive even integers is a 
classification of the finite irreducible 
(i.e., finite non-abelian) 
subgroups of $\GL(2,\C)$.
\end{theorem}

\subsection{Degree 3}
\begin{theorem}\label{Deg3PrimSol}
Let $G$ be a finite solvable primitive subgroup 
of $\GL(3,\C)$. 
\begin{itemize}
\item[{\rm (i)}]
$G\cap \mathrm{D}(3,\C) = Z(G)$ 
has order divisible by $3$. 
\item[{\rm (ii)}]
$G/Z(G)\cong E_H:= (C_3\times C_3)\rtimes H$ 
where $H$ is $C_4$ or $Q_8$ (the quaternion 
group of order $8$) or $\SL(2,3)$. 
\item[{\rm (iii)}] For $n\equiv 0 \bmod 3$, denote
by $m_{n, H}$ the number of $\GL(3,\C)$-conjugacy
classes of $G$ such that $|Z(G)| = n$ and 
$G/Z(G)\cong E_H$.  
If $H= Q_8$ then $m_{n,H}=2$. If $H$ is $C_4$ or 
$\SL(2,3)$ then $m_{n,H}=3$. 
\end{itemize}
\end{theorem}
\begin{proof}
Once more we appeal to our proofs of these results 
for non-modular absolutely irreducible primitive 
groups: see \cite[\S \! 6.4]{FlanneryOBrien}.
\end{proof}

We give a more detailed version of 
Theorem~\ref{Deg3PrimSol}. 
The generating sets below are transcribed from those 
in Theorems~6.22--6.24 and Corollary~6.26 of 
\cite{FlanneryOBrien} (deleting redundant generators).
\begin{theorem}\label{SolvablePrimitive3}
Let $n\equiv 0 \bmod 3$.  Define
\[
u = 
{\small
\frac{1}{\epsilon-\epsilon^2}
\left[
\begin{array}{ccc}
1 & 1 & 1 \\
1 & \epsilon & \epsilon^2 \\
1 & \epsilon^2 & \epsilon
\end{array} \right]} 
\qquad
\mbox{and} 
\qquad 
u' = {\small
\frac{1}{\epsilon-\epsilon^2}
\left[
\begin{array}{ccc}
1 & \epsilon & \epsilon \\
\epsilon^2 & \epsilon & \epsilon^2 \\
\epsilon^2 & \epsilon^2 & \epsilon
\end{array} \right]} 
\]
where $\epsilon = e^{2\pi \mathrm{i}/3}$.
Up to conjugacy, the finite solvable 
primitive subgroups of 
$\GL(3,\C)$ with center of
order $n$ and central quotient $E_H$ are as follows.
\begin{itemize}
\item[{\rm (i)}]  For $H=C_4 \! :$
\[
\begin{array}{l}
\quad 
\langle s, u,  z_n \rangle  
\qquad \quad \mbox{all} \,  n \\
\begin{array}{ll}
\left. 
\begin{array}{l}
\langle s, -u, z_n\rangle  \\
\langle s, \mathrm{i}u, z_n\rangle
\end{array} 
\right\} & n \ \mbox{odd} 
\end{array} \\
\begin{array}{ll}
\left.
\begin{array}{l}
\langle s, z_{4n} u,  z_n 
\rangle \\
\langle s, z_{4n}^2 u,  z_n \rangle
\end{array}
\right\} 
&  n \equiv 0 \bmod 4
\end{array}\\
\begin{array}{ll}
\left. 
\begin{array}{l}
\langle s, \mathrm{i}u, z_n\rangle \\
\langle s, \sqrt{\mathrm{i}} 
\hspace{.5pt} u, z_n \rangle 
\end{array}
\right\} & n\equiv 2 \bmod 4 .
\end{array}
\end{array}
\]
\item[{\rm (ii)}]  For $H = Q_8 \! :$
\[
\begin{array}{ll}
\langle s, u, u', z_n \rangle & \\
\langle s, u, -u', z_n \rangle & \ n 
\ \mbox{odd only} \\ 
\langle s, uz_{2n}, u', z_n \rangle & \ n 
\ \mbox{even only} .
\end{array}
\]
\item[]
\item[{\rm (iii)}] For $H= \SL(2,3) \! :$
\ $\langle u, x_{27}, z_n \rangle$, \,  
$\langle u, z_{3n}x_{27}, z_n \rangle$, \,
$\langle u, z_{3n}^2 x_{27}, z_n  \rangle$.
\end{itemize}
\end{theorem}

It would be pleasing to have classifications
of  finite solvable primitive subgroups of 
$\GL(p,\C)$ for larger $p$.

The finite non-solvable primitive subgroups of 
$\SL(3,\C)$ were also listed by 
Blichfeldt~\cite{Blichfeldt}. 
We fill out this listing to all of $\GL(3,\C)$ 
via the techniques employed in degree $2$ to 
prove Theorem~\ref{Degree2PrimitivesAll}.
\begin{definition}
Let 
\[
a' =   {\small
\frac{1}{2}
\left[
\renewcommand{\arraycolsep}{.15cm}
 \begin{array}{rrr}
-1 & \mu_2 & \mu_1 \\
\mu_2 & \mu_1 & -1 \\
\mu_1 & -1 & \mu_2
\end{array} \right]}
\qquad \text{and} \qquad
b'= 
{\small
\left[
\renewcommand{\arraycolsep}{.15cm}
 \begin{array}{rrr}
-1 & 0 & 0 \\
0 & 0 & -\epsilon^2 \\
0 & -\epsilon & 0
\end{array} \right]},
\]
where $\mu_1 = \frac{-1+\sqrt{5}}{2}$, 
$\mu_2 = \frac{-1-\sqrt{5}}{2}$, and
$\epsilon = e^{2\pi \mathrm{i}/3}$.
Let $c'$ be the $3\times 3$ back-circulant
matrix whose first row is
$\frac{1}{\sqrt{-7}} \,  [\omega^4-\omega^3 ,  
\omega^2 - \omega^5 , \omega - \omega^6 ]$ 
where  
 $\omega = e^{2\pi \mathrm{i}/7}$.
\end{definition}
\begin{theorem}\label{Deg3NonSolPrim}
A finite non-solvable subgroup $G$
of $\GL(3,\C)$, with center of order $n$,
is primitive if and only if $G$ is 
$\GL(3,\C)$-conjugate to one of
\begin{itemize}
\item[{\rm (i)}] $\langle 
s, \mathrm{diag}(1,-1,-1), a' , z_n \rangle 
 \cong \mathrm{Alt}(5)\times Z(G);$
\item[{\rm (ii)}] $\langle s, \mathrm{diag}(1,-1,-1),
a', b', z_n\rangle$ for $n \equiv 0 \bmod 3$,
 containing $3 \cdot \mathrm{Alt}(6)$ and
with central quotient $\mathrm{Alt}(6);$
\item[{\rm (iii)}] 
$\langle \mathrm{diag}(\omega, \omega^2, \omega^4),
c', z_n \rangle \cong \PSL(2,7)\times Z(G)$. 
\end{itemize}
\end{theorem}
\begin{proof}
The possible isomorphism types of 
central quotient, and the matrix generators, 
are apparent from 
\cite[pp.~250--251]{MillerBlichfeldtDickson}.

We sketch a proof of (i) only. 
Let $G$ be a finite subgroup of 
$\GL(3,\C)$ such that $G/Z(G)\cong \mathrm{A}_5$.
Since $|H^2(\mathrm{A}_5, Z(G)) |\leq 2$, and 
the Schur cover $\SL(2,5)$ of $\mathrm{A}_5$ has 
no faithful irreducible ordinary representation of 
degree $3$, $G$ must split over its center. Also, 
$\GL(3,\C)$ does not contain a subgroup with 
central quotient $\mathrm{S}_5$. The hypotheses 
of Corollary~\ref{SingleCClNonsolvable} are 
therefore satisfied. 
Direct computation shows that 
$\langle s, \mathrm{diag}(1,-1,-1), a' \rangle
\cong \mathrm{A}_5$.
\end{proof}
\begin{theorem}
The union of $\mathcal{L}^* \cup \mathcal{M}^*$ 
for $p=3$ (see Theorem{\em ~\ref{SolvableFullMonp}}) 
together with the set of all groups listed in 
Theorems{\em ~\ref{SolvablePrimitive3}} and
{\rm \ref{Deg3NonSolPrim}}
is a classification of the finite irreducible 
subgroups of $\GL(3,\C)$.
\end{theorem}

\section{Verification and access to the lists}
\label{Implementation}

We implemented our classifications in {\sc Magma}:
see \cite{ImplementationURL}.
The input is a positive integer $m$ and 
a prime $p$ dividing $m$; the output is 
a list of irreducible monomial subgroups 
of $\GL(p,\C)$ of order $m$ up to 
$\GL(p,\C)$-conjugacy and their labels. 
The projective family is implemented 
only for $p \leq 11$. 
Other groups are returned for all 
input $m$ and $p$. 

Each output group $G$ is given by a 
generating set of monomial matrices over 
a cyclotomic field determined by $m$.
Currently, such fields can be realized up 
to size $2^{30}$ in {\sc Magma}. Our 
default is the `sparse' option. 
An isomorphic copy of $G$ defined 
over a finite field may be constructed as in 
\cite[\S \! 4.3]{Recognition}, and then
we may use other algorithms 
for finite matrix groups to study $G$.

Hensel lifting (see Lemma~\ref{Hensellike}) 
is done using {\sc Magma} intrinsic 
functions. We could avoid $p$-adic 
polynomial arithmetic by computing 
over residue rings $\Z/q^n\Z$ (for $n$ and 
primes $q$ determined by $m$).

We consider briefly the cost of setting 
up the groups of order $m$ in $\Mo(p,\C)$. 
Timings depend on $m$, 
the number of prime factors of $m$, and 
$v$ (see Notation~\ref{dAndvDefinitions}). 
For many orders, setup takes just a few 
CPU seconds. More expensive examples include 
those where $m=pq$ for a prime $q$  
of order $1$ modulo $p$.
%; this maximizes $v$. 
In Table~\ref{time}, we state the CPU time in 
seconds taken to construct representatives
 of all $t$ classes of order at most $m$ 
for degrees $3$ and $5$.  
We used {\sc Magma} V2.25-2 on a 2.6GHz machine. 

\begin{table}[ht]
      \centering
      \begin{tabular}{c|crr|crr}
\toprule $m$ & $p$ & $t\phantom{tt}$ & Time & 
$p$ & $t\phantom{tt}$ & Time
\\
        \midrule
2000 & 3 & 2229 & 28 & 5 & 373 & 7 \\
4000 & 3 & 4994 & 206 & 5 & 850 & 54 \\
6000 & 3 & 7943 & 778 & 5 & 1328 & 210 \\
8000 & 3 & 10993 & 2033 & 5 & 1892 & 525 \\
10000 & 3 & 14131 & 4711 & 5 & 2445 & 1089 \\ \hline 
      \end{tabular}
\medskip
\caption{Setting up all $t$ classes of order at 
most $m$ in $\Mo(p, \C)$}
\label{time}
\end{table}

Conjugacy class representatives for all finite 
irreducible subgroups of $\GL(2,\C)$ and 
$\GL(3,\C)$ are also available.

\subsection{Checking correctness}
Much data about the groups 
is routinely corroborated using {\sc Magma}.
We can test whether a finite $G \leq \GL(p, \C)$ 
is (absolutely) irreducible. 
By exploiting the isomorphic copy of $G$ defined 
over a finite field, we can check $|G|$ and 
other group-theoretic properties, 
such as the isomorphism type of $G/Z(G)$.
We verified claims for solvable groups of orders 
$m\leq 10^4$ and all $p$ dividing $m$.
Non-solvable groups were checked for $p\leq 11$ 
and up to order $10^6$.

Lemma~\ref{CheckByCharacters}~(ii) 
underpins a rudimentary but effective 
correctness testing procedure, 
which we now summarize.
\begin{enumerate}
\item Fix $m>1$ and a prime $p$ dividing $m$.
\item List the monomial groups of order $m$ and degree
$p$ from our implementation.
\item Partition this list by isomorphism. 
\item For each isomorphism type  
$G$, use the algorithm of \cite{exp-math} 
to construct its inequivalent 
faithful irreducible monomial 
representations in $\GL(p, \C)$.  
Compute the number of $\mathrm{Aut}(G)$-orbits 
in the set of 
equivalence classes.  \label{REF-check}
\item If all groups of 
order $m$ are known, then apply 
step~\eqref{REF-check} to each.
\end{enumerate}
If the lists produced in steps~(3) and (4) 
coincide for every isomorphism type, then we 
have verified that the output from step~(2)
is irredundant. 
If the list produced in step~(5) also 
coincides, then the output is complete.
For non-solvable $G$, the algorithm of 
\cite{exp-math} constructs only those 
representations defined over $\Q$,
so correlation between lists is more limited.

In step~(5), we use the following criterion
to isolate monomial groups.
\begin{lemma}
A finite irreducible solvable subgroup of 
$\GL(p,\C)$ is monomial if and only if it 
has a non-central abelian normal subgroup.
\end{lemma}
\noindent
Note that a finite irreducible monomial subgroup 
of $\GL(p,\C)$ is not isomorphic to any primitive 
subgroup of 
$\GL(p,\C)$~\cite[Theorem~2.15]{Flannery4FiniteFields}.

We applied steps~(2), (3), and (4) of the
 correctness test to solvable monomial groups 
of order at most $10000$. 
The {\sc SmallGroups} library \cite{SmallGroups} 
contains the groups of order at most $2000$ 
(excluding $2^{10}$); step~(5) was applied to all.
We thereby reconciled our results with 
the classification in 
\cite{Conlon} of the finite 
irreducible $p$-subgroups of $\GL(p,\C)$,
and that in \cite{DZII}
of the finite irreducible 
subgroups of $\SL(p,\C)\cap \Mo(p,\C)$. 

An (obvious) variant of the procedure was used 
to check accuracy of the primitive group lists 
from Section~\ref{Deg2and3}.

\subsection{The number of conjugacy classes 
of monomial groups}

Our implementation can simply {\it count} 
the $\GL(p, \C)$-conjugacy classes 
of irreducible subgroups of $\Mo(p, \C)$ 
having order $m$. Since neither fields nor
 generators are constructed, this number 
is computed quickly, even for large $m$.

As an illustration, we counted the conjugacy 
classes of solvable groups of order $m$ up to 
$10^6$ and all $p$ dividing $m$.
We did likewise for non-solvable 
groups in degrees $p \leq 11$. 
Table~\ref{max-classes} shows the 
 orders with the most conjugacy classes
(solvable groups are on the left).

\begin{table}[h]
    
\ \	\ 	\begin{minipage}{.4\textwidth}
      \centering
      \begin{tabular}{lc}
        \toprule Order & No.~classes \\
        \midrule

$2^6\cdot 3^4\cdot 5^2 \cdot 7$ & 684 \\ \hline
$2^4\cdot 3^4 \cdot 7^2 \cdot 13 $ & 648 \\ \hline 
$2^7\cdot 3^4 \cdot 7 \cdot 13 $ & 640 \\ \hline 
$2^4\cdot 3^3 \cdot 5^2 \cdot  7 \cdot 13 $ & 621 \\ \hline 
$2^6\cdot 3^3 \cdot 5 \cdot 7 \cdot 13 $ & 620 \\ \hline 
$2^4\cdot 3^3 \cdot 7 \cdot 13 \cdot 19 $ & 588 \\ \hline 
$2^8\cdot 3^4 \cdot 5 \cdot 7 $ & 585 \\ \hline 
$2^6\cdot 3^5 \cdot 7^2 $ & 573 \\ \hline 
$2^7\cdot 3^3 \cdot 5 \cdot 7^2 $ & 568 \\ \hline 
$2^7\cdot 3^3 \cdot 5^2 \cdot 7 $ & 564 \\ 
        \bottomrule
      \end{tabular}
    \end{minipage}
		\begin{minipage}{.4\textwidth}
      \centering
      \begin{tabular}{lc}
        \toprule Order & No.~classes \\
        \midrule

$2^{4} \cdot 3 \cdot 5^6 $ &   25\\  \hline
$2^3 \cdot 3 \cdot 5^6$ &   25\\  \hline

$2^{11} \cdot 3^2 \cdot 5 \cdot 7$ & 17\\  \hline

$2^{13} \cdot 3 \cdot 5 \cdot 7$ & 16\\  \hline
$2^{12} \cdot 3 \cdot 5 \cdot 7 $ & 16\\  \hline

$2^{10} \cdot 3^3 \cdot 5 \cdot 7$ & 15\\  \hline
$2^{10} \cdot 3^2 \cdot 5 \cdot 7$ & 15\\  \hline

$2^{11} \cdot 3 \cdot 5 \cdot 7$ & 14\\  \hline

$2^{16} \cdot 3 \cdot 5 $ & 13\\  \hline
$2^{15} \cdot 3 \cdot 5 $ & 13\\  
      \bottomrule
      \end{tabular}
    \end{minipage}
		\medskip
\caption{Monomial group orders with the most
$\GL(p,\C)$-conjugacy classes}
\label{max-classes}
\end{table}

\subsection*{Acknowledgments}
We thank Alla Detinko for astute comments.
The second and third authors were supported by 
Marsden Fund of New Zealand grant UOA 1626; 
and by the Hausdorff Research Institute 
for Mathematics, 
as participants in the 2018 
Trimester in Logic and Group Theory.

\bibliographystyle{amsplain}

\end{document}